\def\version{17/12/2014 version 7
\hfill\href{http://arxiv.org/abs/1309.2323}{arXiv:1309.2323} \\
{\small \textsf{To appear in Publications of the Research 
Institute for Mathematical Sciences, Kyoto University}}}

\documentclass[centertags,leqno,a4paper,11pt]{amsart}

\usepackage{hyperref}

\usepackage{amssymb,amscd,amsxtra}
\usepackage{fullpage}
\usepackage{fixltx2e}

\usepackage[all]{xy}
\newdir{ >}{{}*!/-8pt/@{>}}

\usepackage{pigpen}
\def\PO{\text{\pigpenfont R}}

\usepackage{mathrsfs}

\usepackage{braket}

\theoremstyle{plain}
\newtheorem{thm}{Theorem}[section]
\newtheorem{lem}[thm]{Lemma}
\newtheorem{prop}[thm]{Proposition}
\newtheorem{cor}[thm]{Corollary}

\theoremstyle{definition}

\newtheorem{defn}[thm]{Definition}

\numberwithin{equation}{section}

\def\ie{\emph{i.e.}}

\def\:{\colon}
\def\.{\cdot}

\def\<{\left\langle}
\def\>{\right\rangle}
\def\({\left(}
\def\){\right)}
\def\ph#1{\phantom{#1}}
\def\epsilon{\varepsilon}
\def\leq{\leqslant}
\def\geq{\geqslant}

\def\IFF{\quad\Longleftrightarrow\quad}

\def\tilde#1{\widetilde{#1}}
\def\iso{\cong}

\def\CP{\mathbb{C}\mathrm{P}}

\def\F{\mathbb{F}}

\def\N{\mathbb{N}}
\def\R{\mathbb{R}}
\def\RP{\R\mathrm{P}}

\def\Z{\mathbb{Z}}

\DeclareMathOperator{\Ext}{Ext}

\def\id{\mathrm{id}}

\def\op{\mathrm{op}}

\def\phi{\varphi}

\DeclareMathOperator{\Sq}{Sq}

\DeclareMathOperator{\dlQ}{Q}
\def\tQ{\tilde{\dlQ}}
\DeclareMathOperator{\q}{q}
\def\Kriz{{K\v{r}\'{i}\v{z}}}

\DeclareMathOperator{\exc}{exc}
\def\tpsi{\tilde{\psi}}
\def\tPsi{\tilde{\Psi}}
\def\btau{\bar{\tau}}
\def\Comod{\mathbf{Comod}}

\def\DLV{\mathbf{Vect}^{\mathrm{DL}}}
\def\HA{\mathbf{HopfAlg}}

\begin{document}
\title[Power operations and coactions]
{Power operations and coactions in highly commutative
homology theories}
\author{Andrew Baker}
\address{School of Mathematics \& Statistics,
University of Glasgow, Glasgow G12 8QW, Scotland}
\email{a.baker@maths.gla.ac.uk}
\urladdr{http://www.maths.gla.ac.uk/$\sim$ajb}
\thanks{I would like to thank the following for their
interest and helpful comments: Rolf Hoyer, Nick Kuhn,
Tyler Lawson, Peter May, Niko Naumann, Geoffrey Powell,
Birgit Richter, John Rognes and Markus Szymik; also
the referee deserves special thanks for being so
thorough and detecting many errors in formulae. \\[10pt]
This paper is dedicated to the memory of my friend
\textbf{Goro Nishida} (1943--2014), whose pioneering
work on power operations inspired it.
}
\date{\version}
\begin{abstract}
Power operations in the homology of infinite loop spaces,
and $H_\infty$ or $E_\infty$ ring spectra have a long
history in Algebraic Topology. In the case of ordinary
mod~$p$ homology for a prime~$p$, the power operations of
Kudo, Araki, Dyer and Lashof interact with Steenrod operations
via the Nishida relations, but for many purposes this leads 
to complicated calculations once iterated applications of 
these functions are required. On the other hand, the homology 
coaction turns out to provide tractable formulae better suited 
to exploiting multiplicative structure.

We show how to derive suitable formulae for the interaction
between power operations and homology coactions in a wide
class of examples; our approach makes crucial use of modern
frameworks for spectra with well behaved smash products. In
the case of mod~$p$ homology, our formulae extend those of
Bisson and Joyal to odd primes. We also show how to exploit
our results in sample calculations, and produce some apparently
new formulae for the Dyer-Lashof action on the dual Steenrod
algebra.
\end{abstract}

\maketitle

\tableofcontents

\section*{Introduction}

In this note we study the interaction between coactions
over homology Hopf algebroids (such as the Steenrod
algebra for a prime~$p$) and power operations (such as
Dyer-Lashof operations). Some of our results are surely
known, but we are only aware of partial references such
as~\cite{TB&AJ:Nishida,TB&AJ:Qrings} which only deal
with the case of ordinary mod~$2$ homology. In any case,
our approach to understanding this relationship involves
a modern perspective based on a symmetric monoidal
category of spectra with good properties such as that
of~\cite{EKMM}.

The examples we discuss in detail are based on ordinary
mod~$p$ homology for a prime~$p$ and the power operations
originally by Kudo, Araki, Dyer and Lashof, then generalised
by May \emph{et al} \cite{Dyer&Lashof,Kudo&Araki,LNM1176},
usually rather unhistorically referred to as Dyer-Lashof
operations. Studying the interaction between the coaction
and the Dyer-Lashof operations amounts to studying dual
versions of the classical Nishida relations~\cite{Nishida}.
We use knowledge of the coaction of the dual Steenrod
algebra $\mathcal{A}(p)_*$ to investigate the homology
of commutative $S$-algebras $R$ where $\pi_0(R)$ has
characteristic~$p$. Of course such questions were studied
by Steinberger~\cite[chapter~III]{LNM1176}. However, our
approach offers some clarification of the algebra involved
in the Dyer-Lashof action on the dual Steenrod algebra
itself, relating it to work of Kochman~\cite{SOK:DLops}
(see also~\cite{SP:DLops});  our detailed knowledge of
the homology also allows us to give a refined version 
of Steinberger's splitting result giving more information
on the multiplicative structure.

The results of this paper have been used in joint work
with Rolf Hoyer and some were an outcome of discussions
with him.

\subsection*{Notation, etc}
We will use the \emph{floor} and \emph{ceiling} functions
$\lfloor-\rfloor,\,\lceil-\rceil\:\R\to\Z$ taking values
\[
\lfloor x\rfloor = \max\{n\in\Z:n\leq x\},
\quad
\lceil x\rceil = \min\{n\in\Z:n\geq x\}.
\]
In particular, for $x\in\Z$ we have
$\lfloor x\rfloor=\lceil x\rceil=x$, while if $x\notin\Z$,
then $\lceil x\rceil=\lfloor x\rfloor + 1$.

When working with power series $f(t)$ in an indeterminate
$t$, $[f(t)]_{t^n}$ will denote the coefficient of $t^n$
in $f(t)$.

\subsection*{Bimodules}
We will often consider bimodules. If
$\mathrm{R},\mathrm{R}',\mathrm{R}''$ are three rings,
$\mathrm{M}$ is an $\mathrm{R}$-$\mathrm{R}'$-bimodule,
and $\mathrm{N}$ is an $\mathrm{R}'$-$\mathrm{R}''$-bimodule,
then we will denote the tensor product over $\mathrm{R}'$
by $\mathrm{M}\boxtimes_{\mathrm{R}'}\mathrm{N}$. We will
reserve $\otimes_\mathrm{R}$ for the situation where
$\mathrm{R}$ is commutative and $\mathrm{U},\mathrm{V}$
are two left $\mathrm{R}$-modules and denote their tensor
product by $\mathrm{U}\otimes_\mathrm{R}\mathrm{V}$. We
will sometimes consider a \emph{left} $\mathrm{R}_*$-module
$\mathrm{M}_*$ over a graded commutative ring $\mathrm{R}_*$
as having a canonical \emph{right} $\mathrm{R}_*$-module
structure given by
\[
m\. r = (-1)^{|r|\,|m|}rm,
\]
for homogeneous elements $r\in \mathrm{R}_{|r|}$ and
$m\in \mathrm{M}_{|m|}$.

\subsection*{Bigebroids and comodules}

Suppose that $\mathrm{A},\mathrm{B},\mathrm{H}$
are commutative (graded) rings and that
\[
A \xrightarrow{\eta_\mathrm{A}} \mathrm{H}
     \xleftarrow{\eta_\mathrm{B}} \mathrm{B}
\]
are ring homomorphisms. We use these to define
a left $\mathrm{A}$-module structure and a right
$\mathrm{B}$-module structure on $\mathrm{H}$.
Given a right $\mathrm{A}$-module $\mathrm{M}$
and a left $\mathrm{B}$-module $\mathrm{N}$,
we can define the bimodule tensor products
\[
\mathrm{M}\boxtimes_\mathrm{A}\mathrm{H},
\quad
\mathrm{H}\boxtimes_\mathrm{B}\mathrm{N},
\quad
\mathrm{M}\boxtimes_\mathrm{A}\mathrm{H}
           \boxtimes_\mathrm{B}\mathrm{N}.
\]
If $\mathrm{R}$ is a commutative graded ring, then
its \emph{opposite ring} has as its underlying set
$\mathrm{R}^\op = \mathrm{R}$ and multiplication of
homogeneous elements given by
\[
x^\op y^\op = (\pm)(yx)^\op,
\]
where the sign is determined in the usual way in terms
of the degrees of $x,y$. The opposite ring $\mathrm{H}^\op$
admits a right $\mathrm{A}^\op$-module structure and a
left $\mathrm{B}^\op$-module structure and there is a
ring isomorphism
\[
\mathrm{R} \xrightarrow{\iso} \mathrm{R}^\op;
\quad
x \leftrightarrow (\pm)x^\op
\]
which interchange the two pairs of module structures.

\part{Power operations and coactions}\label{part:1}

\section{Extended powers and power operations}\label{sec:ExtPow}

In this section we give some general observations
on extended powers. We will work in the category
$\mathscr{M}_S$ of $S$-modules of~\cite{EKMM} and
write $\wedge$ for $\wedge_S$. For an $S$-module~$M$,
\[
M^{\wedge n} = M\wedge\cdots\wedge M.
\]
For an $S$-module $N$ with a left $\Sigma_n$-action
we will denote the half-smash product by
$E\Sigma_n\ltimes_{\Sigma_n}N$. In particular we will
write
\[
D_nM = E\Sigma_n\ltimes_{\Sigma_n}M^{\wedge n}
\]
for the extended power, and when $G\leq\Sigma_n$,
we will sometimes set
\[
D_GM = E\Sigma_n\ltimes_{G}M^{\wedge n}.
\]
If $M$ is cofibrant then by~\cite[theorem~III.5.1]{EKMM},
the projection of $E\Sigma_n$ to a point induces
a weak equivalence
\[
D_nM = E\Sigma_n\ltimes_{\Sigma_n}M^{\wedge n}
    \xrightarrow{\sim} M^{\wedge n}/\Sigma_n.
\]

More generally, if $R$ is a cofibrant commutative
$S$-algebra, then in the category $\mathscr{M}_R$
of $R$-modules, for an $R$-module $N$ we can define
\[
D^R_nN = E\Sigma_n\ltimes_{\Sigma_n}N^{\wedge_R n},
\]
and if $N$ is a cofibrant $R$-module, the natural
map gives a weak equivalence
\[
D^R_nN = E\Sigma_n\ltimes_{\Sigma_n}N^{\wedge_R n}
           \xrightarrow{\sim} N^{\wedge_R n}/\Sigma_n.
\]
If $M\in\mathscr{M}_S$, there is an isomorphism
\begin{equation}\label{eqn:DnR-RDn}
	R\wedge D_nM \iso D^R_n(R\wedge M).
\end{equation}

Now we recall the definition of power operations.
We will do this in a general setting, for three
commutative $S$-algebras $A,B,E$ (actually, it is
enough to assume that $E$ is an $H_\infty$ ring
spectrum). There is a map $\mu_n\:D_nE \to E$
which induces a diagram of $A$-module morphisms.
\[
\xymatrix{
A\wedge D_nE \ar[rr]^{\;I\wedge\mu_n\;}\ar[dr] && A\wedge E \\
& D^A_n(A\wedge E) \ar[ur]&
}
\]
If $x\:S^m\to A\wedge E$, then the composition of solid
arrows in the commutative diagram
\begin{equation}\label{eq:Theta^e-composition}
\xymatrix{
&&&S^k \ar@{.>}@/_10pt/[dlll]_(.5)e\ar@{.>}@/^10pt/[drrr]^(.5){\Theta^e(x)}&&& \\
A\wedge D_nS^m \ar[rr]^(.45){\id\wedge D_nx} \ar@/_27pt/[rrrrrr]_{\tilde{x}}
   && A\wedge D_n(A\wedge E) \ar[rr]^{I\wedge\mu_n} &&
 A\wedge A\wedge E \ar[rr]^(.55){\mathrm{mult}\wedge\id}
              && A\wedge E \\
              &&&&&&
}
\end{equation}
defines a power operation
\begin{equation}\label{eq:Theta^e}
\Theta^e\:A_m(E)\to A_k(E);
\quad
\Theta^e(x) = \tilde{x}_*e
\end{equation}
for each element $e\in A_k(D_nS^m)=\pi_k(A\wedge D_nS^m)$.

\section{Generalised coactions}\label{sec:Coactions}

For any $S$-module $X$, we can use the unit $S\to B$
and switch maps to induce the horizontal morphisms in
the following commutative diagram.
\begin{equation}\label{eq:B-Unit}
\xymatrix{
A\wedge X \ar[r]^(.4){\iso}\ar[d]_{\mathrm{switch}}
  & A\wedge S\wedge X \ar[r]^{\mathrm{unit}}\ar[d]_{\mathrm{switch}}
  & A\wedge B\wedge X\ar[d]^{\mathrm{switch}} \\
X\wedge A \ar[r]^(.4){\iso} & S\wedge X \wedge A \ar[r]^{\mathrm{unit}}
                        & B\wedge X\wedge A
}
\end{equation}
On applying homotopy $\pi_*(-)$ we obtain maps similar
to coactions, at least when suitable flatness conditions
hold.

We make an algebraic assumption: the left $B_*$-module
$B_*(A)=\pi_*(B\wedge A)$ is flat. Then on passing to
homotopy groups we find that there is an isomorphism
of left $B_*$-modules
\[
B_*(X\wedge A) \xrightarrow{\;\iso\;} B_*(X)\otimes_{B_*} B_*(A),
\]
and an isomorphism
\[
A_*(B\wedge X) \xrightarrow{\;\iso\;} A_*(B)\boxtimes_{B_*}B_*(X).
\]
The rightmost switch map induces an isomorphism
\[
A_*(B\wedge X) \xrightarrow{\;\iso\;} B_*(X)\otimes_{B_*} B_*(A)
\]
which converts the left $A_*$-module structure to a right
module structure. These ingredients give the following
commutative diagram.
\[
\xymatrix{
 & \ar@/_15pt/[ld] A_*(X) \ar@/^15pt/[rd] &  \\
 A_*(B)\boxtimes_{B_*}B_*(X)
   &\ar[l]_(.4){\iso} A_*(B\wedge X)\ar[r]^(.4){\iso}
   & B_*(X)\otimes_{B_*} B_*(A)
}
\]
If $A=B$, then in
\begin{equation}\label{eq:Coaction-twist}
\xymatrix{
 & \ar@/_15pt/[ld]_{\psi} A_*(X) \ar@/^15pt/[rd]^{\tpsi} &  \\
 A_*(A)\boxtimes_{A_*}A_*(X)
   &\ar[l]_(.4){\iso} A_*(A\wedge X)\ar[r]^(.4){\iso}
   & A_*(X)\otimes_{A_*} A_*(A)
}
\end{equation}
the homomorphism $\psi$ is the usual left $A_*(A)$-coaction
on $A_*(X)$, while $\tpsi$ is obtained by composing
$\psi$ with the antipode of the Hopf algebroid $A_*(A)$ and
a switch map. In fact $\tpsi$ is a right coaction
making $A_*(X)$ into a right $A_*(A)$-comodule. If we also
take $E=A$, then for each $e\in A_k(D_nS^m)$ there is a
power operation $\Theta^e$ as in~\eqref{eq:Theta^e}, but
also another obtained by interchanging the r\^oles of the
two factors of $A$,
\begin{equation}\label{eq:twist-Theta^e}
\tilde{\Theta}^e = \chi \Theta^e \chi,
\end{equation}
where $\chi\:A_*(A)\to A_*(A)$ is the antipode induced
by the switch map on $A\wedge A$.

The unit $S\to B$ induces the downward morphisms in
the following commutative diagram.
\begin{equation}\label{eq:CoactionDiag}
\xymatrix@C=0.75cm@R=1.0cm{
  A\wedge S\wedge D_n(A\wedge E)\ar[r]^(.5){\iso}
    & A\wedge D_n(S\wedge A\wedge E) \ar[r]\ar[ddd]
    & A\wedge(S\wedge A\wedge E) \ar[d]_{\iso}\ar[r]
    & A\wedge S\wedge E\ar@/^40pt/[ddd] \\
  A\wedge S\wedge D_nS^m \ar[u]^{\id\wedge\id\wedge D_nx}\ar[d]
    && S\wedge(A\wedge A)\wedge E\ar[d]\ar[r]
    & S\wedge A\wedge E\ar[d]\ar[u]^{\iso} \\
  A\wedge B\wedge D_nS^m \ar[d]_{\id\wedge\id\wedge D_nx}
    && B\wedge(A\wedge A)\wedge E\ar[r]
    & B\wedge A\wedge E\ar[d]_{\iso} \\
  A\wedge B\wedge D_n(A\wedge E)\ar[r]^(.5){\iso}
    & A\wedge D_n^B(B\wedge(A\wedge E)) \ar[r]
    & A\wedge B\wedge(A\wedge E)\ar[u]^{\iso}\ar[r]
    & A\wedge B\wedge E
}
\end{equation}
On applying $\pi_*(-)$ to this diagram we obtain
an algebraic analogue.
\begin{equation}\label{eq:CoactionDiag-pi*}
\xymatrix@C=0.75cm@R=1.0cm{
  A_*(S\wedge D_n(A\wedge E))\ar[r]^(.5){\iso}
    & A_*(D_n(S\wedge A\wedge E)) \ar[r]\ar[ddd]
    & A_*(S\wedge A\wedge E) \ar[d]_{\iso}\ar[r]
    & A_*(S\wedge E)\ar@/^40pt/[ddd] \\
  A_*(S\wedge D_nS^m) \ar[u]^{(\id\wedge D_nx)_*}\ar[d]
    && S_*((A\wedge A)\wedge E)\ar[d]\ar[r]
    & S_*(A\wedge E)\ar[d]\ar[u]^{\iso} \\
  A_*(B\wedge D_nS^m) \ar[d]_{(\id\wedge D_nx)_*}
    && B_*((A\wedge A)\wedge E)\ar[r]
    & B_*(A\wedge E)\ar[d]_{\iso} \\
  A_*(B\wedge D_n(A\wedge E))\ar[r]^(.5){\iso}
    & A_*(D_n^B(B\wedge(A\wedge E))) \ar[r]
    & A_*(B\wedge(A\wedge E))\ar[u]^{\iso}\ar[r]
    & A_*(B\wedge E)
}
\end{equation}

When $B=A$ and $A_*(A)$ is $A_*$-flat, $(A_*,A_*(A))$ has
the structure of a Hopf algebroid. For any spectrum~$X$,
the unit $S\to A$ induces a map
\[
A\wedge X\xrightarrow{\;\iso\;}
            A\wedge S \wedge X \to A\wedge A \wedge X,
\]
and there is a left coaction
\[
\psi\:A_*(X) \to A_*(A)\boxtimes_{A_*}A_*(X)
\]
which fits into the following commutative diagram.
\[
\xymatrix{
& A_*(A\wedge X)\ar[dd]^{\iso} \\
A_*(X) \iso A_*(S\wedge X)\ar@/^15pt/[ur]\ar@/_15pt/[dr]_\psi &  \\
& A_*(A)\boxtimes_{A_*}A_*(X)
}
\]
In this situation,~\eqref{eq:CoactionDiag-pi*}
can be used to study the $A_*(A)$-coaction and
its relationship with power operations defined
above. Taking an element
\[
e\in A_*(D_nS^m)\iso A_*(S\wedge D_nS^m)
\]
and chasing it upwards to the right and downwards
to
\[
A_*(A\wedge E)\iso A_*(A)\boxtimes_{A_*}A_*(E)
\]
and then comparing the result with that obtained
by going downwards to the right,
\[
\xymatrix{
&&& A_*(E)\ar[r]^(.46){\iso}\ar[d]^{\psi}\ar@/_58pt/[dd]_{\tpsi}
    & A_*(S\wedge E) \\
A_*(D_nS^m)\ar@/^15pt/[urrr]\ar[drrr]\ar[d]_{\tpsi}
&&& A_*(A)\boxtimes_{A_*}A_*(E)\ar[r]^(.57){\iso}\ar[d]^{\iso}
  & A_*(A\wedge E)             \\
A_*(D_nS^m)\otimes_{A_*}A_*(A)\ar[rrr] &&&
  A_*(E)\otimes_{A_*}A_*(A)\ar[r]^(.57){\iso}
     & A_*(E\wedge A)
}
\]
we obtain the following important formula,
\begin{equation}\label{eq:Coaction-Powop}
\tpsi(\Theta^e(x)) =
 \sum_i(1\otimes\chi(\theta_i))\Theta^{e_i}(\tpsi x),
\end{equation}
where $\psi(e) = \sum_i\theta_i\otimes e_i$.

\section{Further generalisations}\label{sec:Coactions-modules}

The situation of the previous sections can be generalised
somewhat. Suppose that $M$ is a right $A$-module. Then we
can replace the element of $A_k(D_nS^m)$ with $e\in M_k(D_nS^m)$
and use the composition
\begin{equation}\label{eq:Theta^e-composition-module}
\xymatrix{
M\wedge D_nS^m \ar[rr]_(.45){\id\wedge D_nx} \ar@/^20pt/[rrrrrr]^{\tilde{x}}
   && M\wedge D_n(A\wedge E) \ar[rr]_{I\wedge\mu_n} &&
 M\wedge A\wedge E \ar[rr]_(.55){\mathrm{mult}\wedge\id}
              && M\wedge E
}
\end{equation}
to define a power operation
\begin{equation}\label{eq:Theta^e-module}
\Theta^e\:A_m(E)\to M_k(E);
\quad
\Theta^e(x) = \tilde{x}_*e
\end{equation}
analogous to that of~\eqref{eq:Theta^e}.

In order to get a sensible notion of left coaction
$M_*(X)\to M_*(B)\boxtimes_{B_*}B_*(X)$ leading to
analogues of the formulae above in good situations,
it is necessary to assume that $B_*(A)$ is $B_*$-flat,
and also that one of the following conditions holds:
\begin{itemize}
\item
$A_*(X)$ is $A_*$-flat;
\item
$M_*$ is $A_*$-flat as a right $A_*$-module.
\end{itemize}
When $B=A$, the assumptions that $A_*(A)$ is flat
as a left or right $A_*$-module are equivalent,
and in the most important cases in algebraic
topology this holds for any $M$. We leave the
interested reader to work out the details. Such
operations are likely to be hard to work with
unless~$M$ has suitable multiplicative structure
(e.g., it is a commutative $A$-algebra).

One important class of examples is that where
$A=B=E_n$, the $n$-th Lubin-Tate spectrum for
a prime~$p$, and $M=K_n$, the $n$-th Morava
$K$-theory. In this case, $K_n$ is an $E_n$
ring spectrum (not homotopy commutative if~$p=2$);
more generally, we could take $M=E_n\wedge W$,
where~$W$ is a generalised Moore spectrum as
in~\cite{DGD&TL:proRingObj}. The work of the latter 
suggests defining power operations using pro-systems 
of such operations; this is presumably related to 
the work of McClure~\cite[chapter~IX]{LNM1176} on 
power operations in $K$-theory.

\part{Eilenberg-Mac~Lane spectra and Dyer-Lashof operations}
\label{part:2}

\section{Eilenberg-Mac~Lane spectra and the dual
Steenrod algebra}\label{sec:EM}

In this section we discuss the important case of the
Eilenberg-Mac~Lane spectrum for a prime~$p$ and take
$A=B=H=H\F_p$. The dual Steenrod algebra
$\mathcal{A}_*=\mathcal{A}(p)_*=H_*(H)$ is actually
a Hopf algebra over $\pi_*(H)=\F_p$ since the two
unit homomorphisms coincide. We will usually write
$\otimes=\otimes_{\F_p}$ in place of $\boxtimes_{\F_p}$
as there is no danger of confusion. The above isomorphism
\[
H_*(H)\boxtimes_{\F_p}H_*(X)
     \xrightarrow{\;\iso\;}
         H_*(X)\otimes_{\F_p}H_*(H)
\]
coincides with the composition
$\mathrm{switch}\circ(\chi\otimes I)$, and
\[
\tpsi = \mathrm{switch}\circ(\chi\otimes I)\circ\psi.
\]
On a basic tensor $\alpha\otimes x\in H_*(H)\otimes H_*(X)$
this gives
\[
\alpha\otimes x \longleftrightarrow
     (-1)^{|\alpha|\.|x|}x\otimes\chi(\alpha).
\]

The Steenrod algebra $\mathcal{A}^*$ is the $\F_p$-linear
dual of $\mathcal{A}_*$ with associated dual pairing
\[
\braket{-|-}\:\mathcal{A}^*\otimes\mathcal{A}_*\to\F_p.
\]
This gives rise to a \emph{right}\/ action of $\mathcal{A}^*$
on a \emph{left}\/ $\mathcal{A}_*$-comodule $M_*$ by
\[
a_* x = x\cdot a = (-1)^{|a|\,|x|}\sum_i\braket{a|\gamma_i}x_i,
\]
where $a\in\mathcal{A}^*$, $x\in M_*$ and
$\psi x = \sum_i\gamma_i\otimes x_i$. There is also a
dual pairing
\[
\braket{-|-}\:\mathcal{A}_*\otimes\mathcal{A}^*\to\F_p
\]
defined by
\[
\braket{\alpha|a} = (-1)^{|\alpha|\.|a|} \braket{a|\chi\alpha},
\]
giving an alternative formulation of the right action
as
\[
a_* x = \sum_i\braket{\gamma_i'|a}x_i,
\]
where $\tpsi x = \sum_i x_i\otimes\gamma_i'$.

\subsection{The case $p=2$}\label{subsec:EM-2}

When $A=B=H=H\F_2$, the dual Steenrod algebra is
\[
\mathcal{A}_* = \F_2[\xi_r:r\geq1] = \F_2[\zeta_r:r\geq1],
\]
where the Milnor generator $\xi_r\in\mathcal{A}_{2^r-1}$
is defined to be the image of the generator of
$H_{2^r-1}(\RP^\infty)$ under the homomorphism induced
by the canonical map $\RP^\infty\to\Sigma H\F_2$, and
$\zeta_r=\chi(\xi_r)$ is its conjugate; by convention
$\xi_0=\zeta_0=1$. We will make use of the generating
series
\[
\xi(t) = t + \sum_{r\geq1}\xi_rt^{2^r},
\quad
\zeta(t) = t + \sum_{r\geq1}\zeta_rt^{2^r}
\]
which are composition inverses, i.e.,
$\zeta(\xi(t)) = t = \xi(\zeta(t))$.

We have
\[
H_{2m+r}(D_2S^m) =
\begin{cases}
\F_2 &\text{if $r\geq0$}, \\
\;0 &\text{otherwise},
\end{cases}
\]
and the generator in degree $r+2m$ gives the operation
$\dlQ_r = \dlQ^{r+m}$. We write $\tQ_r = \tQ^{r+m}$
for the twisted version of these as
in~\eqref{eq:twist-Theta^e}, so
\[
\tQ_r = \chi \dlQ_r\chi = \chi \dlQ^{r+m} \chi = \tQ^{r+m}.
\]

\begin{thm}\label{thm:Coaction-DL-2}
Let $x\in H_m(E)$ and $\psi(x) = \sum_i\alpha_i\otimes x_i$.
Then
\[
\sum_{m\leq r}\psi(\dlQ^rx)t^r =
\sum_{m\leq k}\sum_{0\leq j\leq k}\sum_i
           \xi(t)^k\tQ^j\alpha_i\otimes \dlQ^{k-j}x_i,
\]
or equivalently
\[
\psi(\dlQ^rx) =
\sum_{m\leq k}\sum_{0\leq j\leq k}\sum_i
\left[\xi(t)^k\right]_{t^r}\tQ^j\alpha_i\otimes\dlQ^{k-j}x_i.
\]
\end{thm}
\begin{proof}
We recall that for $m\in\Z$, there is a weak
equivalence
\[
D_2S^m \xrightarrow{\;\sim\;}\Sigma^m\RP^\infty_m,
\]
where $\RP^\infty_m$ is the Thom spectrum of the
virtual bundle $m\lambda$, and $\lambda\downarrow\RP^\infty$
is the canonical real line bundle associated to
the real sign representation of $\Sigma_2$. When
$m\geq0$,
\[
\RP^\infty_m = \RP^\infty/\RP^{m-1}.
\]
Writing $\bar{e}_{r+m}$ ($r\geq0$) for the image
of the generator $e_r\in H_r(\RP^\infty)$ under
the Thom isomorphism
\[
H_*(\RP^\infty)\iso H_{*+m}(\RP^\infty_m),
\]
the coaction is given by
\[
\sum_{r\geq0} t^{r+m} \psi\bar{e}_{r+m} =
\sum_{s\geq0} \xi(t)^{s+m}\otimes\bar{e}_{s+m}.
\]
Under the composition of the isomorphisms
\[
H_*(D_2S^m)
  \xrightarrow{\;\iso\;} H_*(\Sigma^m\RP^\infty_m)
  \xrightarrow{\;\iso\;} H_{*-m}(\RP^\infty_m)
\]
induced by the above equivalence, the following
elements
\[
e_r\otimes x_m^{\otimes2}
 \longleftrightarrow \bar{e}_{r+m},
\]
correspond, where $x_m\in H_m(S^m)$ is the generator.
Now the result follows from~\eqref{eq:Coaction-Powop},
which gives the following in terms of the right coaction
$\tpsi$,
\[
\sum_{m\leq r}\tpsi(\dlQ^rx)t^r =
\sum_{m\leq k} \dlQ^k(\tpsi x)(1\otimes\zeta(t)^k).
\qedhere
\]
\end{proof}

We will sometimes use generating functions to express
such formulae. For example, we have the series
\begin{align*}
\dlQ_t &= \sum_{r\in\Z}\dlQ^r \, t^r, \\
\intertext{and on substituting $\zeta(t)$ for $t$,}
\dlQ_{\zeta(t)} &= \sum_{r\in\Z}\dlQ^r \,\zeta(t)^r.
\end{align*}
Then
\begin{equation}\label{eq:tpsiQ_t-2}
\tpsi\dlQ_tx = \sum_{|x|\leq r}\tpsi(\dlQ^rx)\,t^r
= \sum_{|x|\leq r}\dlQ^r(\tpsi x)\,\zeta(t)^r
= \dlQ_{\zeta(t)}(\tpsi x).
\end{equation}

The following formulae for Dyer-Lashof operations
at the prime~$2$ are due to
Steinberger~\cite[theorem~III.2.2]{LNM1176}.
\begin{thm}\label{thm:Steinberger2.2}
For $r,s\geq1$,
\begin{align*}
\dlQ^{2^s-2}\zeta_1 &= \zeta_s, \\
\dlQ^r\zeta_1 &\neq 0, \\
\dlQ^r\zeta_s &=
\begin{cases}
\dlQ^{r+2^s-2}\zeta_1 & \text{\rm if $r\equiv0,-1\pmod{2^s}$}, \\
\ph{abcde}0 & \text{\rm otherwise}.
\end{cases}
\end{align*}
\end{thm}
\begin{cor}\label{cor:Steinberger2.2}
For $s,t\geq1$,
\begin{align*}
\tQ^{2^s-2}\xi_1 &= \xi_s, \\
\tQ^r\xi_1 &\neq 0, \\
\tQ^r\xi_s &=
\begin{cases}
\tQ^{r+2^s-2}\xi_1 & \text{\rm if $r\equiv0,-1\pmod{2^s}$}, \\
\ph{abcde}0 & \text{\rm otherwise}.
\end{cases}
\end{align*}
\end{cor}

For later use we record a result that may be
known but we know of no reference.
\begin{lem}\label{lem:DL-xi}
For $s\geq1$,
\begin{equation}\label{eq:DL-xi}
\dlQ^{2^s}\xi_s = \xi_{s+1} + \xi_1\xi_s^2.
\end{equation}
\end{lem}
\begin{proof}
Before proving this we note that if $1\leq r\leq s$,
then for degree reasons
\begin{align*}
\dlQ^{2^s}(\zeta_{r}\xi_{s-r}^{2^r})
&= (\dlQ^{2^r-1}\zeta_{r})\dlQ^{2^s-2^r+1}(\xi_{s-r}^{2^r})
   + (\dlQ^{2^r}\zeta_{r})\dlQ^{2^s-2^r}(\xi_{s-r}^{2^r}) \\
&= \zeta_{r+1}\xi_{s-r}^{2^{r+1}}.
\end{align*}

Suppose that~\eqref{eq:DL-xi} is true for $s<n$.
By definition of the antipode $\chi$, and using
Theorem~\ref{thm:Steinberger2.2} we obtain
\begin{align*}
\dlQ^{2^n}\xi_n &=
\dlQ^{2^n}(\zeta_n + \zeta_{n-1}\xi_1^{2^{n-1}}
               + \cdots + \zeta_1\xi_{n-1}^2) \\
&= \zeta_{n+1} + (\dlQ^{2^{n-1}}\zeta_{n-1})\xi_1^{2^{n}}
               + \cdots + (\dlQ^2\zeta_1)\xi_{n-1}^{2^2} \\
&= \zeta_{n+1} + \zeta_{n}\xi_1^{2^{n}}
               + \cdots + \zeta_2\xi_{n-1}^{2^2}  \\
&= (\zeta_{n+1} + \zeta_{n}\xi_1^{2^{n}} + \cdots
               + \zeta_2\xi_{n-1}^{2^2} + \zeta_1\xi_{n}^{2})
               + \zeta_1\xi_{n}^{2} \\
&= \xi_{n+1} + \xi_1\xi_{n}^{2}.
\qedhere
\end{align*}
\end{proof}

\subsection{The case of an odd prime}\label{subsec:EM-odd}

Suppose that~$p$ is an odd prime and $A=H=H\F_p$. The
dual Steenrod algebra is
\[
\mathcal{A}_*
= \F_p[\xi_r: r\geq 1]\otimes\Lambda(\tau_s: s\geq0)
= \F_p[\zeta_r: r\geq 1]\otimes\Lambda(\btau_s: s\geq0)
\]
where the Milnor generators $\xi_r\in\mathcal{A}_{2(p^r-1)}$
and $\tau_r\in\mathcal{A}_{2p^r-1}$ are the images of
generators of $H_{2p^r-1}(BC_p)$ and $H_{2p^r}(BC_p)$
under the homomorphism induced by the canonical map
$BC_p\to\Sigma H\F_p$, and $\zeta_r=\chi(\xi_r)$,
$\btau_r=\chi(\tau_r)$ are their conjugates; by
convention $\xi_0=\zeta_0=1$.

For $m\in\Z$,
\[
H_{2mp+2r(p-1)-\epsilon}(D_pS^{2m}) =
\begin{cases}
\F_p &\text{if $r\geq0$ and $\epsilon=0$}, \\
\F_p &\text{if $r\geq1$ and $\epsilon=1$}, \\
\;0 &\text{otherwise},
\end{cases}
\]
and a suitably chosen generator in degree
$2mp+2r(p-1)-\epsilon$ gives rise to the operation
$\beta^\epsilon\dlQ_r = \beta^\epsilon\dlQ^{r+m}$.

In order to give a similar discussion to that for
the case $p=2$, we follow the outline
of~\cite[section~V.2]{LNM1176}. Let
\[
W = \{(x_1,\ldots,x_p)\in\R^p : x_1+\cdots + x_p=0\}
\]
be the reduced real permutation representation of
$\Sigma_p$ in which not all elements act orientably,
although $C_p\leq\Sigma_p$ does act by preserving
orientations. Given any finite dimensional real
vector space $U$, we can view $U^p=U\oplus\cdots\oplus U$
with the permutation action of $\Sigma_p$ as equivalent
to
\[
(\R\oplus W)\otimes_\R U\iso U\oplus(W\otimes_\R U)
\]
with $\Sigma_p$ acting only on the left hand factor
and second summands respectively. As
$C_p$-representations,
\[
W \iso W_1 \oplus\cdots \oplus W_{(p-1)/2},
\]
where $W_r = \R^2$ with the generator of $C_p$ acting
as the matrix
$\begin{bmatrix}
\cos(2\pi r/p) & -\sin(2\pi r/p) \\
\sin(2\pi r/p) & \ph{-}\cos(2\pi r/p) \\
\end{bmatrix}$,
which commutes with
$\begin{bmatrix}
0 & -1 \\
1 & \ph{-}0 \\
\end{bmatrix}$.
Therefore each $W_r$ together with $W$ itself has
a natural complex structure compatible with the
$C_p$-action, so $W\otimes_\R U$ can be viewed as
a complex $C_p$-representation. In particular, for
any~$n$, as a $C_p$-representation,
\[
(\R^n)^p \iso \R^n \oplus (W\otimes_\R\R^n)
         \iso \R^n \oplus W^n,
\]
where $W^n$ has the componentwise action, and it
follows that
\[
D_{C_p}S^n \iso
  S^n\wedge E\Sigma_p\ltimes_{C_p}(W^n)^\dag,
\]
where $(-)^\dag$ denotes one-point compactification
of a vector space. Here $E\Sigma_p\ltimes_{C_p}(W^n)^\dag$
is the Thom spectrum of the bundle
\[
E\Sigma_p\ltimes_{C_p}W^n \downarrow BC_p.
\]
As explained in~\cite[section~V.2]{LNM1176}, this
spectrum can be interpreted as the suspension spectrum
of a truncated lens space, but the orientability of
this bundle suffices for our purposes since there
is a Thom isomorphism in mod~$p$ homology
\[
H_*(BC_p) \iso
H_{*+n(p-1)}(E\Sigma_p\ltimes_{C_p}(W^n)^\dag).
\]
We remark that this viewpoint is likely to be useful
in investigating the kind of operations mentioned in
Section~\ref{sec:Coactions-modules} associated with
Lubin-Tate spectra and Morava $K$-theory.

%

Let $z\in H^1(BC_p)$ and let $y=\beta z\in H^2(BC_p)$
be generators of
\[
H^*(BC_p) = \F_p[y]\otimes\Lambda(z).
\]
Let $a_n\in H_n(BC_p)$ be dual to
$z^{\epsilon(n)}y^{\lfloor n/2\rfloor}$, where
$\epsilon(n) = (1-(-1)^n)/2$ and $\lfloor-\rfloor$
is the floor function. We will use a formula for the
left coaction
$\psi\:H_*(BC_p)\to\mathcal{A}_*\otimes H_*(BC_p)$,
originally due to Milnor~\cite{JWM:StAlg}, see also
Boardman's account~\cite{JMB:8-foldway}.

We introduce two formal variables $t_+,t_-$ in degrees
$-2,-1$ respectively (so the usual graded commutativity
rules apply), and defining generating series
\begin{align*}
a(t) = a(t_+,t_-) &=
\sum_{\substack{n\geq1 \\ r=0,1}}a_{2n-r}t_+^{n-r}t_-^r, && \\
\xi(t) = \xi(t_+) &= \sum_{r\geq0} \xi_r t_+^{p^r}, &
\zeta(t) = \zeta(t_+) &= \sum_{r\geq0} \zeta_r t_+^{p^r}, \\
\tau(t) = \tau(t_+,t_-) &=  t_- + \sum_{r\geq0} \tau_r t_+^{p^r}, &
\btau(t) = \btau(t_+,t_-) &=  t_- + \sum_{r\geq0} \btau_r t_+^{p^r}.
\end{align*}
Notice that $\tau(t)$ and $\btau(t)$ have odd degree
and so $\tau(t)^2=0=\btau(t)^2$. The left coaction
is given by
\begin{align}
\psi a(t) &=
\sum_{\substack{n\geq1 \\ r=0,1}}\psi a_{2n-r}t_+^{n-r}t_-^r \notag\\
&= a(\xi(t_+),\tau(t_+,t_-)) \notag\\
&=
\sum_{k\geq1}
\biggl(\xi(t_+)^k \otimes a_{2k}
-\tau(t_+,t_-)\xi(t_+)^{k-1} \otimes a_{2k-1}
\biggr)                             \notag \\
&=
\sum_{k\geq1}
\left(\xi(t_+)^k \otimes a_{2k}
- \sum_{r\geq0}
    \tau_rt_+^{p^r}\xi(t_+)^{k-1}\otimes a_{2k-1}
+ \xi(t_+)^{k-1} \otimes a_{2k-1}t_-\right)
                                       \notag \\
&=
\sum_{k\geq1}
\left(\xi(t_+)^k \otimes a_{2k}
- \sum_{r\geq0}
    \tau_rt_+^{p^r}\xi(t_+)^{k-1}\otimes a_{2k-1}
- \xi(t_+)^{k-1}t_- \otimes a_{2k-1}\right).
                             \label{eq:Coaction-odd}
\end{align}
Notice the effect of interchanging $t_-$ and $a_{2k-1}$
which disappears when we instead take the right coaction:
\begin{align}
\tpsi a(t) &= a(\zeta(t_+),\btau(t_+,t_-)) \notag\\
&=
\sum_{k\geq1}
\biggl( a_{2k}\otimes\zeta(t_+)^k
+ a_{2k-1} \otimes\btau(t_+,t_-)\zeta(t_+)^{k-1}
\biggr)                \notag \\
&=
\sum_{k\geq1}
\left(a_{2k}\otimes \zeta(t_+)^k
+ \sum_{r\geq0}
a_{2k-1}\otimes\btau_rt_+^{p^r}\zeta(t_+)^{k-1}
+ a_{2k-1}\otimes \zeta(t_+)^{k-1}t_- \right).
                          \label{eq:R-Coaction-odd}
\end{align}
By comparing coefficients of monomials in $t_+$ and $t_-$
we obtain explicit formulae for $n\geq1$:
\begin{subequations}\label{eq:Coaction-odd-explicit}
\begin{align}
\psi(a_{2n})
&=
\sum_{k=1}^{n}\left[\xi(t_+)^k\right]_{t_+^{n}}\otimes a_{2k}
 - \sum_{\substack{0\leq p^r\leq n \\ 1\leq k\leq n}}
\left[\tau_rt_+^{p^r}\xi(t_+)^{k-1}\right]_{t_+^n}\otimes a_{2k-1},
\label{eq:Coaction-odd-explicit-even} \\
\psi(a_{2n-1}) &=
\sum_{k=1}^n\left[\xi(t_+)^{k-1}\right]_{t_+^{n-1}}\otimes a_{2k-1},
\label{eq:Coaction-odd-explicit-odd}
\end{align}
\end{subequations}
or equivalently
\begin{subequations}\label{eq:R-Coaction-odd-explicit}
\begin{align}
\tpsi(a_{2n})
&=
\sum_{k=1}^{n}a_{2k}\otimes\left[\zeta(t_+)^k\right]_{t_+^{n}}
+ \sum_{\substack{0\leq p^r\leq n \\ 1\leq k\leq n}}
a_{2k-1}\otimes \left[\btau_rt_+^{p^r}\zeta(t_+)^{k-1}\right]_{t_+^n},
\label{eq:R-Coaction-odd-explicit-even} \\
\tpsi(a_{2n-1}) &=
\sum_{k=1}^na_{2k-1}\otimes \left[\zeta(t_+)^{k-1}\right]_{t_+^{n-1}}.
\label{eq:R-Coaction-odd-explicit-odd}
\end{align}
\end{subequations}

We recall from \cite[chapter~III]{LNM1176} the following
definitions of Dyer-Lashof operations on $x\in H_m(E)$.
As for the prime~$2$, $\dlQ^rx$ and $\beta\dlQ^rx$
originate on elements $H_*(D_{C_p}S^m)$, namely whenever
$2r\geq m$,
\begin{align*}
(-1)^r\nu(m)a_{(2r-m)(p-1)}\otimes x_m^{\otimes p}
                &\longmapsto\dlQ^rx, \\
(-1)^r\nu(m)a_{(2r-m)(p-1)-1}\otimes x_m^{\otimes p}
                &\longmapsto \beta\dlQ^rx.
\end{align*}
where
\[
\nu(m) = (-1)^{m(m-1)(p-1)/4}\biggl(((p-1)/2)!\biggr)^m.
\]
Notice that this factor does not depend on~$r$.

We will use generating series in the indeterminates
$t_+,t_-$ for encoding actions of Dyer-Lashof operations.
We set
\begin{align*}
\dlQ_t(-) = \dlQ_{t_+,t_-}(-)
&= \sum_{r}\dlQ^r(-)\,t_+^{r(p-1)},  \\
\beta\dlQ_t(-) = \beta\dlQ_{t_+,t_-}(-)
&= \sum_{r}\beta\dlQ^r(-)\,t_+^{r(p-1)-1}t_-,
\end{align*}
where the coefficients are operators that can be
applied to homology elements.

We obtain the following result on the coaction and
Dyer-Lashof operations in the homology of an $H_\infty$
ring spectrum $E$. To ease the notation, we state it
in terms of the right coaction $\tpsi$, omitting
$\otimes$ when no confusion seems likely.

Choose $\omega\in\F_{p^2}^\times$ to be a primitive
$(p-1)$-th root of~$-1$; although not uniquely
determined, $\omega$ gives a well defined element
of the cyclic group $\F_{p^2}^\times/\F_{p}^\times$.
We make use of the ceiling function $\lceil-\rceil$.
\begin{thm}\label{thm:CoactionDLp}
For $x\in H_m(E)$,
\begin{align*}
\tpsi(\dlQ_t x)
=&
\sum_{\lceil m/2\rceil\leq r}\tpsi(\dlQ^rx)t_+^{r(p-1)} \\
=&
\dlQ_{\omega\zeta(\omega^{-1}t_+)}(\tpsi x) \;+ \\
& \quad\quad\quad 
(-1)^m\biggl(
\beta\dlQ_{\omega\zeta(\omega^{-1}t_+),
   \btau(\omega\zeta(\omega^{-1}t_+),\omega^{-1}t_-)}(\tpsi x)
- \beta\dlQ_{\omega\zeta(\omega^{-1}t_+),t_-}(\tpsi x)
\biggr)   \\
=&
\sum_{\lceil m/2\rceil\leq k}\dlQ^k(\tpsi x)
\biggl(1\otimes(\omega\zeta(\omega^{-1}t_+))^{k(p-1)}\biggr) \\
&\quad\quad\quad
+ \sum_{r\geq0} \sum_{\lceil m/2\rceil\leq\ell}
(-1)^{m+r}\beta\dlQ^\ell(\tpsi x)
\biggl(
1\otimes\btau_rt_+^{p^r}(\omega\zeta(\omega^{-1}t_+))^{\ell(p-1)-1}
\biggr), \\[10pt]
\tpsi(\beta\dlQ_t(x))
=&
\sum_{\lceil m/2\rceil\leq r}\tpsi(\beta\dlQ^rx)t_+^{r(p-1)-1}t_- \\
=&\;\;\;
\beta\dlQ_{\omega\zeta(\omega^{-1}t_+),t_-}(\tpsi x) \\
=&
\sum_{\lceil m/2\rceil\leq s}\beta\dlQ^s(\tpsi x)
\biggl(1\otimes(\omega\zeta(\omega^{-1}t_+))^{s(p-1)-1}t_-\biggr).
\end{align*}
\end{thm}
\begin{proof}[Outline of proof for $m\geq0$]
Using the description of an extended power $D_{C_p}S^m$
as a suspension of a truncated lens space, we can pull
back to $BC_p$. The origins of $\dlQ^rx,\beta\dlQ^rx$
then lies in the $m$-fold suspension of the elements
\[
(-1)^ra_{2(p-1)r} = (\omega^{-1})^{r(p-1)}a_{2(p-1)r},
\quad
(-1)^ra_{2(p-1)r-1} = (\omega^{-1})^{r(p-1)}a_{2(p-1)r-1}.
\]
As the map defining the Dyer-Lashof operations factors
through $D_{C_p}S^m\to D_pS^m$, the formulae follow
{}from~\eqref{eq:R-Coaction-odd-explicit}.

The case where $m<0$ can be proved using Thom spectra
of virtual bundles.
\end{proof}

The odd primary part of
Steinberger~\cite[theorem~III.2.2]{LNM1176} gives
the following result.
\begin{thm}\label{thm:Steinberger2.2-p}
For $r,s\geq1$,
\begin{align*}
\dlQ^{(p^s-1)/(p-1)}\tau_0 &= (-1)^s\btau_s, \\
\beta\dlQ^{(p^s-1)/(p-1)}\tau_0 &= (-1)^s\chi\xi_s
                                 = (-1)^s\zeta_s, \\
\beta\dlQ^r\tau_0 &\neq 0, \\
\dlQ^r\zeta_s &=
\begin{cases}
\ph{+1}(-1)^{s}\beta\dlQ^{r+(p^s-1)/(p-1)}\tau_0
                    & \text{\rm if $r\equiv-1\pmod{p^s}$}, \\
(-1)^{s+1}\beta\dlQ^{r+(p^s-1)/(p-1)}\tau_0
                    & \text{\rm if $r\equiv0\pmod{p^s}$}, \\
\ph{abcde}0 & \text{\rm otherwise},
\end{cases} \\
\dlQ^r\btau_s &=
\begin{cases}
(-1)^{s+1}\dlQ^{r+(p^s-1)/(p-1)}\tau_0
     & \text{\rm if $r\equiv0\pmod{p^s}$}, \\
\ph{abcde}0 & \text{\rm otherwise}.
\end{cases}
\end{align*}
In particular,
\[
\dlQ^{p^s}\zeta_s = \zeta_{s+1},
\qquad
\dlQ^{p^s}\btau_s = \btau_{s+1}.
\]
\end{thm}

Our next result is analogous to Lemma~\ref{lem:DL-xi}.
\begin{lem}\label{lem:DL-xi-tau-p}
For $r\geq0$ and $s\geq1$,
\begin{subequations}\label{eq:DL-xi-tau-p}
\begin{align}
\label{eq:DL-xi-tau-p-1}
\dlQ^{p^r}\tau_r &= \tau_{r+1} - \tau_0\xi_{r+1}, \\
\label{eq:DL-xi-tau-p-2}
\beta\dlQ^{p^r}\tau_r &= \xi_{r+1}, \\
\label{eq:DL-xi-tau-p-3}
\dlQ^{p^s}\xi_s &= \xi_{s+1} - \xi_1\xi_s^p.
\end{align}
\end{subequations}
\end{lem}
\begin{proof}
Conjugation is defined by the recursive formulae
\begin{align*}
\btau_s + \btau_{s-1}\xi_1^{p^{s-1}} + \btau_{s-2}\xi_2^{p^{s-2}}
+ \cdots + \btau_0\xi_s + \tau_s &= 0, \\
\zeta_s + \zeta_{s-1}\xi_1^{p^{s-1}} + \zeta_{s-2}\xi_2^{p^{s-2}}
+ \cdots + \zeta_1\xi_{s-1}^{p} + \xi_s &= 0.
\end{align*}
Applying $\dlQ^{p^s}$ to the first equation, using
the Cartan formula and considering degrees carefully,
we obtain
\begin{align*}
\dlQ^{p^s}\tau_s &=
- \biggl(\dlQ^{p^s}\btau_s
+ (\dlQ^{p^{s-1}}\btau_{s-1})\xi_1^{p^{s}}
+ (\dlQ^{p^{s-2}}\btau_{s-2})\xi_2^{p^{s-1}}
+ \cdots + (\dlQ^1\btau_0)\xi_s^p\biggr) \\
&= - \biggl(\btau_{s+1}
+ \btau_{s}\xi_1^{p^{s}}
+ \btau_{s-1}\xi_2^{p^{s-1}}
+ \cdots + \btau_1\xi_s^p\biggr) \\
&= - \biggl(\btau_{s+1}
+ \btau_{s}\xi_1^{p^{s}}
+ \btau_{s-1}\xi_2^{p^{s-1}}
+ \cdots + \btau_1\xi_s^p + \btau_0\xi_{s+1}
+ \tau_{s+1}\biggr)
+ \btau_0\xi_{s+1} + \tau_{s+1} \\
&= \tau_{s+1} - \tau_0\xi_{s+1}, \\
\intertext{and}
\dlQ^{p^s}\xi_s &=
- \biggl(\dlQ^{p^s}\zeta_s
+ (\dlQ^{p^{s-1}}\zeta_{s-1})\xi_1^{p^{s}}
+ (\dlQ^{p^{s-2}}\zeta_{s-2})\xi_2^{p^{s-1}}
+ \cdots + (\dlQ^p\zeta_1)\xi_s^{p^2}\biggr) \\
&= - \biggl(\zeta_{s+1}
+ \zeta_{s}\xi_1^{p^{s}}
+ \zeta_{s-1}\xi_2^{p^{s-1}}
+ \cdots + \zeta_2\xi_s^{p^2}\biggr) \\
&= - \biggl(\zeta_{s+1}
+ \zeta_{s}\xi_1^{p^{s}}
+ \zeta_{s-1}\xi_2^{p^{s-1}}
+ \cdots + \zeta_2\xi_s^{p^2}
+ \zeta_1\xi_s^p + \xi_{s+1}\biggr)
+ \zeta_1\xi_s^p + \xi_{s+1} \\
&= \xi_{s+1} - \xi_1\xi_s^p.
\end{align*}
We also have
\begin{align*}
\beta\dlQ^{p^s}\tau_s &=
   \beta\tau_{s+1} - \beta(\tau_0\xi_{s+1}) \\
&= 0 - (\beta\tau_0)\xi_{s+1} + \tau_0(\beta\xi_{s+1}) \\
&= \xi_{s+1}, \\
\end{align*}
since the Bockstein $\beta$ acts on $\mathcal{A}_*$
by the left action of $\mathcal{A}^*$, i.e., if
$a\in\mathcal{A}_*$ and $\psi(a)=\sum_ia'_i\otimes a''_i$,
then
\[
\beta a =
 \sum_i\langle\beta,\chi(a'_i)\rangle a''_i,
\]
where $\langle-,-\rangle$ is the dual pairing between
$\mathcal{A}^*$ and $\mathcal{A}_*$. This gives
$\beta\tau_0=-1$ as used above.
\end{proof}

\section{Dyer-Lashof operations on the dual Steenrod algebra}
\label{sec:DL-A*}

We will require more information on the action of Dyer-Lashof
operations in $\mathcal{A}(p)_*$. In the discussion following
we make use of Kochman~\cite{SOK:DLops} and Steinberger~\cite{LNM1176}.

Let $\mathrm{R}$ be a commutative ring. Define the \emph{Newton
polynomials}
\[
\mathrm{N}_n(t) =
\mathrm{N}_n(t_1,\ldots,t_n) \in \mathrm{R}[t_1,\ldots,t_n]
\]
recursively by setting $\mathrm{N}_1(t)=t_1$ and
\[
\mathrm{N}_n(t) =
t_1\mathrm{N}_{n-1}(t) - t_2\mathrm{N}_{n-2}(t)
+ \cdots +
(-1)^{n-2}t_{n-1}\mathrm{N}_{1}(t)
+ (-1)^{n-1}nt_{n}.
\]
It is well known that for a prime~$p$,
\[
\mathrm{N}_{pn}(t) \equiv \mathrm{N}_n(t)^p \pmod{p},
\]
In $\mathcal{A}(p)_*$, we can consider the values
of these obtained by setting
\[
t_n =
\begin{cases}
\xi_r & \text{if $n=p^r-1$}, \\
\;0 & \text{otherwise},
\end{cases}
\]
and we denote these elements by $\mathrm{N}_{n}(\xi)$.
They satisfy recurrence relations of the form
\[
\mathrm{N}_{n}(\xi) =
-\xi_1\mathrm{N}_{n-p+1}(\xi) - \xi_2\mathrm{N}_{n-p^2+1}(\xi)
+ \cdots
\]
and in particular
\begin{align*}
\mathrm{N}_{p^s-1}(\xi) &=
-\xi_1\mathrm{N}_{p^s-p}(\xi) - \xi_2\mathrm{N}_{p^s-p^2}(\xi)
+ \cdots - \xi_{s-1}\mathrm{N}_{p^s-p^{s-1}}(\xi) - (p^s-1)\xi_s \\
&=
-\xi_1\mathrm{N}_{p^{s-1}-1}(\xi)^p - \xi_2\mathrm{N}_{p^{s-2}-1}(\xi)^{p^2}
+ \cdots - \xi_{s-1}\mathrm{N}_{p-1}(\xi)^{p^{s-1}} + \xi_s.
\end{align*}
Since $\mathrm{N}_{p-1}(\xi)=\xi_1=-\zeta_1$, it follows
that the negatives $-\mathrm{N}_{p^r-1}(\xi)$ satisfy the
same recurrence relation as the conjugates $\zeta_r=\chi(\xi_r)$,
hence for each $s\geq1$,
\begin{equation}\label{eq:Nn(xi)}
\mathrm{N}_{p^s-1}(\xi) = -\zeta_s.
\end{equation}
See~\cite[lemma~2.8]{HW-StAlg} for a closely
related result which also implies this one. We
also mention another easy consequence of the
recursion formula which can be verified by
working modulo the ideal
$(\xi_i:i\geq2)\lhd\mathcal{A}(p)_*$.
\begin{lem}\label{lem:Nnnot0}
For any prime~$p$ and any $k\geq1$,
\[
\mathrm{N}_{k(p-1)}(\xi) \neq 0.
\]
\end{lem}

The generating series for the $(-1)^n\mathrm{N}_n(\xi)$
satisfies the relation
\[
\biggl(1+\sum_{r\geq1}\xi_rt^{p^r-1}\biggr)
   \biggl(\sum_{n\geq1}(-1)^n\mathrm{N}_n(\xi)t^n\biggr)
= \sum_{r\geq1}\xi_rt^{p^r-1},
\]
hence
\begin{equation}\label{eq:Nn(xi)-genfunc}
\sum_{n\geq1}(-1)^n\mathrm{N}_n(\xi)t^n =
1 - \biggl(1+\sum_{r\geq1}\xi_rt^{p^r-1}\biggr)^{-1}
= 1- \frac{t}{\xi(t)}.
\end{equation}

We will give formulae for the $\mathrm{N}_n(\xi)$
modulo the ideal~$(\zeta_j:j\neq s)\lhd\mathcal{A}(p)_*$,
for some fixed $s\geq1$. The recursive formula
for the antipode of $\mathcal{A}(p)_*$ gives
\[
\xi_{ns} \equiv -\zeta_s\xi_{(n-1)s}^{p^s}
                \mod{(\zeta_j:j\neq s)},
\]
and an induction shows that
\begin{equation}\label{eq:xismodxetas}
\xi_{ns} \equiv (-1)^n\zeta_s^{(p^{ns}-1)/(p^s-1)}
                \mod{(\zeta_j:j\neq s)}.
\end{equation}
Combining this with~\eqref{eq:Nn(xi)-genfunc}
we obtain
\begin{equation}\label{eq:Nn(xi)-genfunc-modzetas}
\sum_{n\geq1}(-1)^n\mathrm{N}_n(\xi)t^n \equiv
1 -
\biggl(1+
\sum_{r\geq1}(-1)^r\zeta_s^{(p^{rs}-1)/(p^s-1)}t^{p^{rs}-1}\biggr)^{-1}
    \mod{(\zeta_j:j\neq s)}.
\end{equation}

Our next result is a number theoretic observation.
\begin{lem}\label{lem:BinomCoeff}
Let $p$ be a prime and let $s\geq1$. Suppose that
the natural number $n$ has $p$-adic expansion
\[
n = n_kp^k + n_{k+1}p^{k+1} + \cdots + n_{k+\ell}\,p^{k+\ell}
\]
where $\ell\geq0$ and $n_k,n_{k+\ell}\not\equiv0\pmod{p}$.
Then
\[
\dbinom{np^s-1}{n}\not\equiv0\pmod{p}
\IFF
n_{s+k} \leq n_k-1,\; n_{s+k+1} \leq n_{k+1},
\;\ldots, \; n_{k+\ell} \leq n_{k+\ell-s}.
\]
\end{lem}

\begin{proof}
The $p$-adic expansion of $np^s-1$ is
\begin{multline*}
np^s-1 =
(p-1) + (p-1)p + \cdots + (p-1)p^{s+k-1} \\
+ (n_k-1)p^{s+k} + n_{k+1}p^{s+k+1}
+ \cdots + n_{k+\ell}\,p^{s+k+\ell}
\ph{n_{k+\ell}\,p^{s+k+\ell}}
\end{multline*}
so
\[
\binom{np^s-1}{n} \equiv
\binom{p-1}{n_k}\cdots\binom{p-1}{n_{s+k-1}}
\binom{n_k-1}{n_{s+k}}
\binom{n_{k+1}}{n_{s+k+1}}\cdots\binom{n_{k+\ell-s}}{n_{k+\ell}}
\pmod{p}.
\]
This does not vanish mod~${p}$ exactly when
the stated conditions hold.
\end{proof}

For example, when $p=2$ and $s=1$,
\begin{equation}\label{eq:BinomCoeffvanishing-2}
\dbinom{2n-1}{n}\not\equiv0\pmod{2}
\IFF
\text{$n = 2^k$ for some $k\geq0$}.
\end{equation}

We will use this in proving our next result.
\begin{lem}\label{lem:N-InverseSeries}
Let $p$ be a prime and let $s\geq1$. Then
\begin{multline*}
\biggl(
1 +
\sum_{m\geq1}(-1)^m\zeta_s^{(p^{ms}-1)/(p^s-1)}t^{p^{ms}-1}
\biggr)^{-1} \\
\equiv
1 +
\sum_{\substack{n\in\N \\ \binom{np^s-1}{n}\not\equiv0\pmod{p}}}
                  (-1)^{n+1}\binom{np^s-1}{n}\zeta_s^{n}t^{n(p^s-1)}
                      \mod{(\zeta_j:j\neq s)}.
\end{multline*}
Hence
\[
\mathrm{N}_{n(p^s-1)}(\xi) \equiv
\dbinom{np^s-1}{n}\zeta_s^{n}
             \mod{(\zeta_j:j\neq s)},
\]
and this is non-zero precisely when the
coefficients in the $p$-adic expansion
\[
n=n_kp^k+\cdots+n_{k+\ell}\,p^{k+\ell}
\]
satisfy the inequalities
\[
n_{s+k} \leq n_k-1,\; n_{s+k+1} \leq n_{k+1},
\;\ldots, \; n_{k+\ell} \leq n_{k+\ell-s}.
\]
\end{lem}
\begin{proof}
We will use residue calculus to determine the
coefficient of $t^n$ of positive degree, and
will denote by
\[
\oint f(z)\,dz = c_{-1}
\]
the coefficient of $z^{-1}$ in a meromorphic 
Laurent series
\[
f(z) = \sum_{k_0\leq k\in\Z} c_kz^k
     \in\mathrm{R}[\![z]\!][z^{-1}],
\]
where $\mathrm{R}$ is any commutative ring and
$k_0\in\Z$. We may apply standard rules of Calculus
for manipulating such expressions. For example,
on changing variable by setting
$z=h(w)\in\mathrm{R}[\![w]\!][w^{-1}]$, we obtain
\[
\oint f(z)\,dz = \oint f(h(w))h'(w)\,dw.
\]

We may determine the coefficient of $t^{n(p^s-1)}$
in $t/\xi(t)$ by calculating
\begin{align*}
\oint \frac{t}{\xi(t)}\,\frac{dt}{t^{n(p^s-1)+1}}
&= \oint\frac{\zeta(u)}{u}\,\frac{du}{\zeta(u)^{n(p^s-1)+1}} \\
&= \oint\biggl(\frac{\zeta(u)}{u}\biggr)^{-n(p^s-1)}\,\frac{du}{u^{n(p^s-1)+1}} \\
&\equiv \oint(1+\zeta_su^{p^s-1})^{-n(p^s-1)}\,\frac{du}{u^{n(p^s-1)+1}}  \\
&\equiv \oint(1+\zeta_sv)^{-n(p^s-1)}\,\frac{(-1)dv}{v^{n+1}} \\
&\equiv -\binom{-n(p^s-1)}{n}\zeta_s^n  \\
&\equiv (-1)^{n+1}\binom{n(p^s-1)+(n-1)}{n}\zeta_s^n \\
&\equiv (-1)^{n+1}\binom{np^s-1}{n}\zeta_s^n
                           \mod{(\zeta_j:j\neq s)}.
\end{align*}
Now we can use Lemma~\ref{lem:BinomCoeff} to complete
the analysis of these coefficients.
\end{proof}


To determine the Dyer-Lashof operations on $\mathcal{A}_*$
we use results of Kochman~\cite{SOK:DLops}, where
the operations are calculated in $H_*(BO;\F_2)$ and
$H_*(BU;\F_p)$. We remark that in each case the Thom
isomorphism is known to respect the Dyer-Lashof
operations, so this determines the Dyer-Lashof actions
in $H_*(MO;\F_2)$ and $H_*(MU;\F_p)$.

For $p=2$, $H_*(MO)=H_*(MO;\F_2)$ is the polynomial
algebra on generators $a_n\in H_n(MO)$ which correspond
to generators of $H_*(BO)$ coming from those in
$H_*(BO(1))=H_*(\RP^\infty)$. Under the homomorphism
$H_*(MO)\to\mathcal{A}_*$ induced by the orientation
$MO\to H\F_2$,
\[
a_n \mapsto
\begin{cases}
\xi_s & \text{if $n=2^s-1$}, \\
0     & \text{otherwise}.
\end{cases}
\]
The Newton polynomial
$\mathrm{N}_n(a)=\mathrm{N}_n(a_1,\ldots,a_n)\in H_n(MO)$
corresponds to the Hopf algebra primitive generator
in $H_n(BO)$, so~\cite[corollary~35]{SOK:DLops} gives
\[
\dlQ^r\mathrm{N}_n(a)
         = \binom{r-1}{n-1}\mathrm{N}_{n+r}(a).
\]
This yields the following formula in $\mathcal{A}_*$:
\begin{equation}\label{eq:DLop-Nn(xi)-2}
\dlQ^r\mathrm{N}_n(\xi)
         = \binom{r-1}{n-1}\mathrm{N}_{n+r}(\xi).
\end{equation}
Using~\eqref{eq:Nn(xi)} for $p=2$, we obtain
\begin{equation}\label{eq:DLop-zeta-2}
\dlQ^r\zeta_s
         = \binom{r-1}{2^s-2}\mathrm{N}_{2^s-1+r}(\xi),
\end{equation}
and it is easy to see that
\begin{equation}\label{eq:Qr-SOK-2}
\dbinom{r-1}{2^s-2}\equiv1\pmod{2}
\IFF
r\equiv0,-1\pmod{2^s}.
\end{equation}
Using Lemma~\ref{lem:Nnnot0}, this recovers part
of Steinberger's result, see our
Theorem~\ref{thm:Steinberger2.2}.

For an odd prime~$p$, $H_*(MU)$ is polynomial on
generators $b_n\in H_{2n}(MU)$ coinciding under
the Thom isomorphism with generators of $H_*(BU)$
coming from $H_*(BU(1))=H_*(\CP^\infty)$. Under
the homomorphism induced by the orientation
$MU\to H\F_p$,
\[
b_n \mapsto
\begin{cases}
\xi_s & \text{if $n=p^s-1$}, \\
0     & \text{otherwise},
\end{cases}
\]
and so by~\cite[theorem~5]{SOK:DLops}, the
Newton polynomial $\mathrm{N}_n(\xi)$ satisfies
\begin{equation}\label{eq:DLop-Nn(xi)-odd}
\dlQ^r\mathrm{N}_n(\xi) =
(-1)^{r+n}\binom{r-1}{n-1}
              \mathrm{N}_{n+r(p-1)}(\xi),
\end{equation}
and using~\eqref{eq:Nn(xi)} we obtain
\begin{equation}\label{eq:DLop-zeta-odd}
\dlQ^r\zeta_s =
(-1)^{r+1}\binom{r-1}{p^s-2}
         \mathrm{N}_{p^s-1+r(p-1)}(\xi).
\end{equation}
It is easy to see that
\begin{equation}\label{eq:Qr-SOK-odd}
\dbinom{r-1}{p^s-2}\not\equiv0\pmod{p}
\IFF
r\equiv0,-1\pmod{p^s},
\end{equation}
thus recovering part of Steinberger's result
(see Theorem~\ref{thm:Steinberger2.2-p}).

%

\section{Verification of the Nishida relations}
\label{sec:NishidaRelns}

For completeness we show how the usual Nishida
relations are consequences of coaction formulae.

\subsection{The case $p=2$}\label{subsec:NishidaRelns-2}

First we recall that with respect to the monomial
basis for $\mathcal{A}_*=\mathcal{A}(2)_*$, the
dual element of $\xi_1^{r_1}\cdots\xi_\ell^{r_\ell}$
is $\Sq^{(r_1,\ldots,r_\ell)}\in\mathcal{A}^*$.
The dual of $\xi_1^n$ is the Steenrod operation
$\Sq^{(n)}=\Sq^n$, i.e.,
\[
\braket{\Sq^n | \xi_1^{r_1}\cdots\xi_\ell^{r_\ell}}
=
\begin{cases}
 1 & \text{if $\xi_1^{r_1}\cdots\xi_\ell^{r_\ell}=\xi_1^n$}, \\
 0 & \text{otherwise}.
\end{cases}
\]
In terms of the right pairing, this becomes
\[
\braket{\zeta_1^{r_1}\cdots\zeta_\ell^{r_\ell} | \Sq^n}
=
\begin{cases}
 1 & \text{if $\zeta_1^{r_1}\cdots\zeta_\ell^{r_\ell}=\zeta_1^n$}, \\
 0 & \text{otherwise}.
\end{cases}
\]
We will work with the right coaction so the
latter formulae will often be used.

Notice that for a left $\mathcal{A}_*$-comodule
$M_*$ and $x\in M_*$, we have
\begin{align}
\psi(x) &=
\sum_{(r_1,\ldots,r_\ell)}
\xi_1^{r_1}\cdots\xi_\ell^{r_\ell}
           \otimes\Sq^{(r_1,\ldots,r_\ell)}_*x,
                    \label{eq:psi-dual-2}\\
\tpsi(x) &=
\sum_{(r_1,\ldots,r_\ell)}
\Sq^{(r_1,\ldots,r_\ell)}_*x
           \otimes\zeta_1^{r_1}\cdots\zeta_\ell^{r_\ell}.
           \label{eq:tpsi-dual-2}
\end{align}
In particular
\[
\Sq^n_*x = \braket{\Sq^n\otimes1 | \psi x}
         = \braket{\tpsi x | 1\otimes\Sq^n}.
\]

We want to determine expressions of the form
$\Sq^r_*\dlQ^sx$ where $x\in H_*(E)$ for a
commutative $S$-algebra~$E$. We have
\[
\Sq^r_*\dlQ^sx = \braket{\tpsi\dlQ^s x | 1\otimes\Sq^r},
\]
and combining these for all values of $s$ we
obtain
\begin{align*}
\Sq^r_*\dlQ_t x =
\sum_{s}(\Sq^r_*\dlQ^sx) t^s
&= \Braket{ \sum_{s}\tpsi\dlQ^sx t^s | 1\otimes\Sq^r} \\
&= \Braket{ \sum_{k}\dlQ^k(\tpsi x) \zeta(t)^k | 1\otimes\Sq^r} \\
&= \Braket{ \dlQ_{\zeta(t)}(\tpsi x) | 1\otimes\Sq^r}.
\end{align*}

In the expression~\eqref{eq:tpsi-dual-2}, applying
$\dlQ_{\zeta(t)}$ to a term yields
\begin{align*}
\dlQ_{\zeta(t)}(\Sq^{(r_1,\ldots,r_\ell)}_*x
           \otimes\zeta_1^{r_1}\cdots\zeta_\ell^{r_\ell})
&=
\dlQ_{\zeta(t)}(\Sq^{(r_1,\ldots,r_\ell)}_*x)
\otimes
\dlQ_{\zeta(t)}(\zeta_1^{r_1}\cdots\zeta_\ell^{r_\ell}) \\
&=
\dlQ_{\zeta(t)}(\Sq^{(r_1,\ldots,r_\ell)}_*x)
\otimes
(\dlQ_{\zeta(t)}\zeta_1)^{r_1}\cdots(\dlQ_{\zeta(t)}\zeta_\ell)^{r_\ell},
\end{align*}
so we need to investigate the terms $\dlQ_{\zeta(t)}\zeta_s$.
In fact for our purposes it is sufficient to know these
$\bmod{(\zeta_j:j>1)}$.
\begin{lem}\label{lem:Qzeta(t)zeta_s-2}
For $s\geq2$,
\[
\dlQ_t\zeta_s \equiv 0 \mod{(\zeta_j:j>1)}.
\]
\end{lem}
\begin{proof}
By~\eqref{eq:Qr-SOK-2} and~\eqref{eq:DLop-zeta-2},
$\dlQ^r\zeta_s\neq0$ only when $r\equiv0,-1\pmod{2^s}$,
and then .
\[
\dlQ^r\zeta_s =
\begin{cases}
\mathrm{N}_{1+2+\cdots+2^{s-1}+2^k+r_{k+1}2^{k+1}+\cdots+2^\ell}(\xi)
          & \text{if $r\equiv0\pmod{2^s}$}, \\
\mathrm{N}_{2+\cdots+2^{s-1}+2^k+r_{k+1}2^{k+1}+\cdots+2^\ell}(\xi)
                  & \text{if $r\equiv-1\pmod{2^s}$},
\end{cases}
\]
for some $k,\ell$ with $s\leq k\leq\ell$. In either case
we find that $\dlQ^r\zeta_s\equiv0\bmod{(\zeta_j:j>1)}$
by using Lemma~\ref{lem:N-InverseSeries} (with $s=1$).
\end{proof}

For the case $s=1$, we have
\[
\dlQ_t\zeta_1 = \frac{1}{t}-\frac{1}{\xi(t)} + \zeta_1,
\]
hence
\begin{align*}
\dlQ_{\zeta(t)}\zeta_1
&= \frac{1}{\zeta(t)}-\frac{1}{t} + \zeta_1 \\
&\equiv \frac{1 - (1+\zeta_1 t) + \zeta_1(t+\zeta_1 t^2)}{(t+\zeta_1 t^2)} \\
&\equiv \frac{\zeta_1^2 t}{(1+\zeta_1 t)} \\
&\equiv \zeta_1^2 t(1+\zeta_1t)^{-1} \mod{(\zeta_j:j>1)}.
\end{align*}
So we have
\begin{align*}
\Sq^r_*\dlQ_t x
&=
\sum_{j\geq0}
     \Braket{\dlQ_{\zeta(t)}(\Sq^j_*x)(\dlQ_{\zeta(t)}\zeta_1)^j | 1\otimes\Sq^r} \\
&=
\sum_{j\geq0}\sum_k
\Braket{\zeta(t)^k(\dlQ_{\zeta(t)}\zeta_1)^{j} | 1\otimes\Sq^r}\dlQ^k\Sq^j_*x \\
&=
\sum_{j\geq0}\sum_k
\Braket{\zeta_1^{2j}t^{j+k}(1+\zeta_1 t)^{k-j} | 1\otimes\Sq^r}\dlQ^k\Sq^j_*x \\
&=
\sum_{j\geq0}\sum_k \binom{k-j}{r-2j}\dlQ^k\Sq^j_*x\,t^{r+k-j},
\end{align*}
or equivalently
\[
\Sq^r_*\dlQ^nx =
   \sum_{j\geq0} \binom{n-r}{r-2j}\dlQ^{n-r+j}\Sq^j_*x,
\]
which is the usual form of the Nishida relations.

\subsection{The case $p$ odd}\label{subsec:NishidaRelns-odd}

We begin by determining formulae for Dyer-Lashof operations
in $\mathcal{A}_*=\mathcal{A}(p)_*$ $\bmod{(\zeta_j:j\geq2)}$.
By~\eqref{eq:DLop-zeta-odd} and~\eqref{eq:Qr-SOK-odd} we
find that for an indeterminate~$t$ of degree~$-2$,
\begin{align*}
\dlQ_t\zeta_s &=
\sum_{k\geq1}
\biggl(
(-1)^{k+1}\binom{kp^s-1}{p^s-2}\mathrm{N}_{(kp^s+p^{s-1}+\cdots+p+1)(p-1)}(\xi)\,t^{kp^s(p-1)}\biggr. \\
& \ph{\binom{kp^s-1}{p^s-2}\mathrm{N}_{(kp^s+p^{s-1}+\cdots+p+1)(p-1)}}
\biggl. +
(-1)^{k}\binom{kp^s-2}{p^s-2}\mathrm{N}_{(kp^s+p^{s-1}+\cdots+p)(p-1)}(\xi)\,t^{(kp^s-1)(p-1)}
\biggr)                     \\
&\equiv
\sum_{k\geq1}
(-1)^{k}(\mathrm{N}_{(kp^{s-1}+p^{s-2}+\cdots+p+1)(p-1)}(\xi))^p\,t^{(kp^s-1)(p-1)}
\mod{(\zeta_j:j\geq2)},
\end{align*}
where the first term in each summand vanishes
thanks to Lemma~\ref{lem:N-InverseSeries}. Also,
when $s\geq2$, Lemma~\ref{lem:N-InverseSeries}
implies that
\begin{equation}\label{eq:Qtzetas=0}
\dlQ_t\zeta_s \equiv 0 \mod{(\zeta_j:j\geq2)}.
\end{equation}
When $s=1$,
\begin{align*}
\dlQ_t\zeta_1
&\equiv
\sum_{k\geq1}
(-1)^{k}(\mathrm{N}_{k(p-1)}(\xi))^p\,t^{(kp-1)(p-1)} \\
&\equiv
t^{-(p-1)}\biggl(\sum_{k\geq1}
\mathrm{N}_{k(p-1)}(\xi)\,(-t^{(p-1)})^k\biggr)^p  \\
&\equiv
t^{-(p-1)}\biggl(\sum_{k\geq1}
\mathrm{N}_{k(p-1)}(\xi)\,(\omega^{-1}t)^{k(p-1)}\biggr)^p
\mod{(\zeta_j:j\geq2)},
\end{align*}
where $\omega\in\F_{p^2}^\times$ is a primitive
$(p-1)$-th root of~$-1$ as introduced earlier.
Using~\eqref{eq:Nn(xi)-genfunc}, we obtain
\begin{equation}\label{eq:Qtzeta1-odd}
\dlQ_{t}\zeta_1 \equiv
-\frac{1}{(\omega^{-1}t)^{p-1}}
\biggl(1-\frac{\omega^{-1}t}{\xi(\omega^{-1}t)}\biggr)^p
\mod{(\zeta_j:j\geq2)}.
\end{equation}
Replacing $t$ by $\omega\zeta(\omega^{-1}t)$ gives
another useful formula:
\begin{align}\label{eq:Qzeta(t)zeta1-odd}
\dlQ_{\omega\zeta(\omega^{-1}t)}\zeta_1
&\equiv
-\frac{1}{\zeta(\omega^{-1}t)^{p-1}}
\biggl(1-\frac{\zeta(\omega^{-1}t)}{\omega^{-1}t}\biggr)^p
                                   \notag\\
&\equiv
\zeta_1^pt^{(p-1)^2}(1-\zeta_1t^{p-1})^{1-p} \\
&\equiv
\sum_{k\geq0}\binom{p-2+k}{k}\zeta_1^{p+k}t^{(p+k-1)(p-1)}
\mod{(\zeta_j:j\geq2)}. \notag
\end{align}

Now we  follow a similar line of argument to that for
the case $p=2$ above. We recall that $\mathcal{A}(p)_*$
has a basis consisting of monomials
\[
\xi_1^{r_1}\cdots\xi_k^{r_k}\tau_0^{e_0}\cdots\tau_\ell^{e_\ell}
\]
where $e_i=0,1$ and $r_i\geq0$. In $\mathcal{A}(p)^*$,
the dual basis element is
$\mathcal{P}^{(r_1,\ldots,r_k;e_0,\ldots,e_\ell)}$. In
particular, $\mathcal{P}^{(r)} = \mathcal{P}^r$ is the
reduced power operation.

We want to determine the
series $\mathcal{P}^r_*\dlQ_tx$, and this turns out
to be given by
\begin{align*}
\mathcal{P}^r_*\dlQ_tx
&= \sum_{j\geq0}
\Braket{\dlQ_{\omega\zeta(\omega^{-1}t)}(\mathcal{P}^j_*x)
(\dlQ_{\omega\zeta(\omega^{-1}t)}\zeta_1)^j | \mathcal{P}^r}  \\
&=
\sum_{j\geq0}\sum_{k}
\Braket{
(\omega\zeta(\omega^{-1}t))^{k(p-1)}
(\dlQ_{\omega\zeta(\omega^{-1}t)}\zeta_1)^j
| \mathcal{P}^r}\dlQ^k\mathcal{P}^j_*x  \\
&=
\sum_{j\geq0}\sum_{k}
t^{(k+j(p-1))(p-1)}\Braket{
\zeta_1^{jp}
(1-\zeta_1t^{p-1})^{(k-j)(p-1)}
| \mathcal{P}^r} \dlQ^k\mathcal{P}^j_*x \\
&=
\sum_{j\geq0}\sum_{k}
t^{(k+j(p-1))(p-1)}
(-1)^{r-jp}\binom{(k-j)(p-1)}{r-jp}t^{(r-jp)(p-1)}
\dlQ^k\mathcal{P}^j_*x \\
&=
\sum_{j\geq0}\sum_{k}
(-1)^{j+r} t^{(k-j+r)(p-1)}
\binom{(k-j)(p-1)}{r-jp}
\dlQ^k\mathcal{P}^j_*x.
\end{align*}
Taking the coefficient of $t^{s(p-1)}$ by putting
$k=s-r+j$ we obtain
\[
\mathcal{P}^r_*\dlQ^sx =
\sum_{j\geq0}(-1)^{r+j}
\binom{(s-r)(p-1)}{r-jp}\dlQ^{s-r+j}\mathcal{P}^j_*x,
\]
which is the usual form of the Nishida relations.

We leave the interested reader to perform a similar
verification of the Nishida relations for
$\mathcal{P}^r_*\beta\dlQ^s$.

\section{Working modulo squares and Milnor primitives}
\label{sec:modsquares}

In this section we work at the prime~$2$, but there
are analogous results at odd primes. Let $E$ be a
commutative $S$-algebra.

As another example of the utility of our methods,
we will investigate the induced coaction
\[
\xymatrix{
H_*(E) \ar[r]_(.4){\psi}\ar@/^20pt/[rr]^{\Psi}
  & \mathcal{A}_*\otimes H_*(E) \ar[r]
  & \mathcal{E}_*\otimes H_*(E)
}
\]
where
\[
\mathcal{E}_* = \mathcal{A}_*/\!/\mathcal{A}_*^{(2)}
              = \mathcal{A}_*/(\zeta_s^2:s\geq1)
\]
is the exterior quotient Hopf algebra dual to the
sub-Hopf algebra of $\mathcal{A}^*$ generated by
the Milnor primitives $\q^r\in\mathcal{A}^{2^{r+1}-1}$
recursively defined by setting $\q^0=\Sq^1$ and
for $r\geq1$,
\[
\q^r = [\q^{r-1},\Sq^{2^r}]
    = \q^{r-1}\Sq^{2^r} + \Sq^{2^r}\q^{r-1}.
\]
To avoid cumbersome notation, we will write $u\circeq v$
in place of $u\equiv v\bmod{(\zeta_s^2:s\geq1)}$ when
working with the quotient ring $\mathcal{E}_*$, and
identify elements of $\mathcal{A}_*$ with their residue
classes.

As with $\psi$, there is a corresponding right
coaction
\[
\xymatrix{
H_*(E) \ar[r]_(.4){\tpsi}\ar@/^20pt/[rr]^{\tPsi}
  & H_*(E)\otimes\mathcal{A}_* \ar[r]
  & H_*(E)\otimes\mathcal{E}_*
}
\]

Our interest is in the general form of the right
coaction on elements of the form $\dlQ^Iz$, or
equivalently in $\q^r_*\dlQ^Iz$ for $r\geq0$. Of
course it is well known that
\[
\q^0_*\dlQ^az = \Sq^1_*\dlQ^az = (a+1)\dlQ^{a-1}z.
\]
Using the monomial basis in the residue classes
$\bar{\xi}_i$ of the $\xi_i$, $\q^r$ is dual to
the residue class of
\[
\xi_{r+1} = \chi(\zeta_{r+1})
= \zeta_1\xi_{r}^2 + \cdots+ \zeta_{r}\xi_1^{2^{r}}
  + \zeta_{r+1}
\equiv \zeta_{r+1} \mod{(\zeta_i^2:i\geq1)}.
\]
Hence $\xi_{r+1}\circeq\zeta_{r+1}$ is primitive
in the quotient Hopf algebra $\mathcal{E}_*$. To
calculate $\q^r_*w$ we may use the formulae
\[
\q^r_*w = \braket{\q^r\otimes1|\Psi w}
        = \braket{\tPsi w|1\otimes\q^r}.
\]

It is clear that the action of the Dyer-Lashof
operations descends from $\mathcal{A}_*$ to the
quotient $\mathcal{E}_*$. We start by determining
the image of $\dlQ^r\zeta_s$ in $\mathcal{E}_*$.
\begin{lem}\label{lem:dlQ^rzetas}
For $s\geq1$ and $r\geq s$, we have
\[
\dlQ^r\zeta_s \circeq
\begin{cases}
\zeta_{s+m} & \text{\rm if $r=2^{s+m}-2^s$}, \\
\;\;\;0 & \text{\rm otherwise}.
\end{cases}
\]
\end{lem}
\begin{proof}
Using \eqref{eq:DLop-zeta-2} and \eqref{eq:Qr-SOK-2},
it suffices to consider the cases $r=2^sk,2^sk-1$.
The Newton recurrence formula gives
\begin{align*}
\dlQ^{2^sk}\zeta_s &= \mathrm{N}_{2^s(k+1)-1}(\xi) \\
&= \xi_1\mathrm{N}_{2^s(k+1)-2}(\xi)
+ \xi_2\mathrm{N}_{2^s(k+1)-4}(\xi) + \cdots \\
&= \xi_1\mathrm{N}_{2^{s-1}(k+1)-1}(\xi)^2
+ \xi_2\mathrm{N}_{2^{s-2}(k+1)-1}(\xi)^4 + \cdots
\end{align*}
and this is $0\bmod{(\zeta_i^2:i\geq1)}$ unless
$2^s(k+1)=2^{s+m}$, \ie, $k=2^m-1$, and then
\[
\dlQ^{2^{s+m}-2^s}\zeta_s \circeq \zeta_{s+m}.
\]
Also,
\[
\dlQ^{2^sk-1}\zeta_s = \mathrm{N}_{2^s(k+1)-2}(\xi)
                     = \mathrm{N}_{2^{s-1}(k+1)-1}(\xi)^2
                     \circeq 0.
\qedhere
\]
\end{proof}

Using the notation
\begin{align*}
\Xi(r) &= \sum_{r+1\leq k}\zeta_k t^{2^k-2^r}, \\
\Xi(r,s) &= \Xi(r) - \Xi(s)
          = \sum_{r+1\leq k\leq s}\zeta_k t^{2^k-2^r},
\end{align*}
where $0\leq r<s$, we obtain the following succinct
formula:
\begin{equation}\label{eq:dlQ^rzetas}
\dlQ_t\zeta_s \circeq
   \sum_{s+1\leq k} \zeta_k t^{2^k-2^s} = \Xi(s).
\end{equation}
If $s\geq1$,
\[
\Xi(s)^2 \circeq 0,
\]
hence when $s_1<s_2$,
\begin{equation}\label{eq:Xi-product}
\Xi(s_1)\Xi(s_2) \circeq \Xi(s_1,s_2)\Xi(s_2),
\end{equation}
we can derive another useful formula. If
$s_1<s_2<\cdots<s_\ell$, then
\begin{multline}\label{eq:dlQ^rzetas-prod}
\dlQ_t(\zeta_{s_1}\zeta_{s_2}\cdots\zeta_{s_\ell}) =
\dlQ_t(\zeta_{s_1})\dlQ_t(\zeta_{s_2})\cdots\dlQ_t(\zeta_{s_\ell})
\circeq
\Xi(s_1,s_2)\Xi(s_2,s_3)\cdots\Xi(s_{\ell-1},s_\ell)\Xi(s_\ell).
\end{multline}

Now we can give a formula for the $\mathcal{E}_*$-coaction.
\begin{prop}\label{prop:tpsi(dlQtz)}
If $z\in H_n(E)$, then
\[
\tPsi\dlQ_tz \circeq
\dlQ_t(\tPsi z) +
\sum_{a\geq n}\sum_{j\geq1}(a+1)\dlQ^{a-2^j+1}(\tPsi z)\zeta_j t^a.
\]
Equivalently, for each $a\geq n$,
\[
\tPsi\dlQ^a z \circeq
\dlQ^a(\tPsi z) + (a+1)\sum_{j\geq1}\dlQ^{a-2^j+1}(\tPsi z)\zeta_j.
\]
\end{prop}
\begin{proof}
This follows from the calculation
\begin{align*}
\tPsi\dlQ^a z
&\circeq \sum_{n\leq k\leq a}\dlQ^k(\tPsi z)
                     \left[(1+\Xi(0))^k\right]_{t^{a-k}} \\
&\circeq \dlQ^a(\tPsi z) +
\sum_{n\leq k\leq a-1}k\dlQ^k(\tPsi z)\left[\Xi(0)\right]_{t^{a-k}} \\
&\circeq \dlQ^a(\tPsi z) +
\sum_{n\leq a-2^j+1\leq a-1}(a-2^j+1)\dlQ^{a-2^j+1}(\tPsi z)\left[\Xi(0)\right]_{t^{2^j-1}} \\
&\circeq \dlQ^a(\tPsi z) +
\sum_{2\leq 2^j\leq a-n+1}(a+1)\dlQ^{a-2^j+1}(\tPsi z)\zeta_j.
\qedhere
\end{align*}
\end{proof}

We can use this to derive formulae for the action
of the Milnor primitives.
\begin{prop}\label{prop:q^s*Q^a}
If $z\in H_n(E)$, $s\geq0$ and $a>n$, then
\[
\q^s_*\dlQ^az =
(a+1)\dlQ^{a-2^{s+1}+1}z +
 \sum_{0\leq r\leq s-1}\dlQ^{a-2^{s+1}+2^{r+1}}(\q^r_*z).
\]
\end{prop}
\begin{proof}
We can determine $\q^s_*\dlQ^az$ using the inner
product, \ie,
\[
\q^s_*\dlQ^az = \braket{\tPsi\dlQ^a z|1\otimes\q^s}.
\]
By Proposition~\ref{prop:tpsi(dlQtz)}, we have
\[
\q^s_*\dlQ^az =
\braket{\dlQ^a(\tPsi z)|1\otimes\q^s}
              + (a+1)\dlQ^{a-2^{s+1}+1}z.
\]
To analyse $\braket{\dlQ^a(\tPsi z)|1\otimes\q^s}$,
we note that only the term in $\dlQ^a(\tPsi z)$ of
form $(?)\otimes\zeta_{s+1}$ can provide a non-zero
contribution, while in $\tPsi z$ any term of form
$(?)\otimes\zeta_{i_1}\cdots\zeta_{i_\ell}$ with
$\ell>1$ contributes zero. Since
$\dlQ^{2^{s+1}-2^{r+1}}(\zeta_{r+1})\circeq\zeta_{s+1}$
we must have
\begin{align*}
\braket{\dlQ^a(\tPsi z)|1\otimes\q^s}
&=
\sum_{r}\braket{\dlQ^a(\q^r_*z\otimes\zeta_r)|1\otimes\q^s} \\
&=
\sum_{r}\braket{\dlQ^{a-2^{s+1}+2^{r+1}}(\q^r_*z)
                              \otimes\zeta_{s+1}|1\otimes\q^s} \\
&= \sum_{r}\dlQ^{a-2^{s+1}+2^{r+1}}(\q^r_*z).
\qedhere
\end{align*}
\end{proof}

This result is useful when calculating with iterated
Dyer-Lashof operations. For example,
\begin{align*}
\q^1_*\dlQ^az &= (a+1)\dlQ^{a-3}z + \dlQ^{a-2}(\q^0_*z), \\
\q^1_*\dlQ^a\dlQ^bz &= (a+1)\dlQ^{a-3}\dlQ^bz + (b+1)\dlQ^{a-2}\dlQ^{b-1}z.
\end{align*}
In general, $\q^s_*\dlQ^{a_1}\cdots\dlQ^{a_{s+1}}z$
does not depend on the coaction on~$z$.

\part{Free commutative $S$-algebras}\label{part:3}

\section{Free commutative $S$-algebras and their homology}
\label{sec:PX}

Following \cite{EKMM} we work in the model categories
of left $S$-modules $\mathscr{M}=\mathscr{M}_S$ and
commutative $S$-algebras $\mathscr{C}=\mathscr{C}_S$.
The latter are the commutative monoids in $\mathscr{M}$.
The forgetful functor $\mathbb{U}\:\mathscr{C}\to\mathscr{M}$
has a left adjoint $\mathbb{P}\:\mathscr{M}\to\mathscr{C}$,
the free commutative $S$-algebra functor, giving a Quillen
adjunction.
\[
\xymatrix{
{\mathscr{C}} \ar@/_8pt/[rr]_{\mathbb{U}}
      && {\mathscr{M}} \ar@/_8pt/[ll]_{\mathbb{P}}
}
\]
For an $S$-module $X$,
\[
\mathbb{P}X = \bigvee_{j\geq 0} X^{(j)}/\Sigma_j,
\]
where $X^{(j)} = X\wedge \cdots \wedge X$ is the $j$-fold
smash power with its evident $\Sigma_j$-action, and
$X^{(j)}/\Sigma_j$ is the orbit spectrum. When $X$ is
cofibrant, the natural map $D_jX \to X^{(j)}/\Sigma_j$
is a weak equivalence, hence there is a weak equivalence
\[
\bigvee_{j\geq0}D_jX \xrightarrow{\;\sim\;} \mathbb{P}X.
\]

The mod~$p$ homology of extended powers $D_nX$ has been
studied extensively, and the answer is expressible in
terms of a free algebra construction. Recently, Kuhn \&
McCarty~\cite{NJK&JMcC:HF2InfLoopSpcs} gave an explicit
description for the prime~$2$, and we adopt a similar
viewpoint. Older references of relevance are
May~\cite{JPM:HomOps}, McClure~\cite[theorem~IX.2.1]{LNM1176},
and Kuhn~\cite{NJK:Transfer}. In keeping with our emphasis
on coactions and comodule structures, we phrase this in
terms of the dual Steenrod algebra, thus avoiding the
local finiteness condition for actions of the Steenrod
algebra.

Fix a prime~$p$ and let $\mathcal{A}_*=\mathcal{A}(p)_*$.
We adopt the following notation.
\begin{itemize}
\item
$\Comod_{\mathcal{A}_*}$ is the category of $\Z$-graded
right $\mathcal{A}_*$-comodules, where we denote the
coaction by $\Psi\:M_*\to M_*\otimes\mathcal{A}_*$.
\item
$\DLV$ is the category of graded $\F_p$-vector spaces
$V_*$ equipped with actions of Dyer-Lashof operations
$\dlQ^r\:V_*\to V_{*+2(p-1)r}$ and
$\beta\dlQ^r\:V_*\to V_{*+2(p-1)r-1}$ (when $p=2$,
$\dlQ^r\:V_*\to V_{*+r}$) subject to the Adem relations
and the unstable condition $\dlQ^r v = 0$ if $2r<|v|$
(when $p=2$, $\dlQ^r v = 0$ if $r<|v|$).
\item
$\Comod_{\mathcal{A}_*}^{\mathrm{DL}}$ is the full
subcategory of $\Comod_{\mathcal{A}_*}\cap\DLV$ which
consists of right $\mathcal{A}_*$-comodules with
Dyer-Lashof action that satisfies the formulae of
Theorem~\ref{thm:CoactionDLp} when~$p$ is odd, or
Theorem~\ref{thm:Coaction-DL-2} when~$p=2$.
\end{itemize}

The free algebra $\mathbb{P}X$ has a natural homotopy
coproduct
$\Delta\:\mathbb{P}X\to\mathbb{P}X\wedge\mathbb{P}X$
induced by the pinch map $X\to X\vee X$. The induced
homomorphism
\[
\Delta_*\:H_*(\mathbb{P}X)
   \to H_*(\mathbb{P}X)\otimes H_*(\mathbb{P}X)
\]
turns $H_*(\mathbb{P}X)=H_*(\mathbb{P}X;\F_p)$ into
cocommutative coalgebra, and so $H_*(\mathbb{P}X)$
is a bicommutative Hopf algebra. This structure is
discussed in detail in~\cite[section~2.3]{NJK&JMcC:HF2InfLoopSpcs}
at least for the prime~$2$. The component maps of
$\Delta$ are transfers associated to inclusions of
block subgroups $\Sigma_r\times\Sigma_s\leq\Sigma_{r+s}$
and the Dyer-Lashof operations on $H_*(\mathbb{P}X)$
satisfy a Cartan formula making it a bicommutative
$\mathcal{A}_*$-comodule Hopf algebra with Dyer-Lashof
action satisfying the restriction condition
$\dlQ^{|x|/2}x=x^p$ if $|x|$ is even (and $\dlQ^{|x|}x=x^2$
if $p=2$). We denote the category of all such bicommutative
Hopf algebras by $\HA^{\mathcal{A}_*,\,\mathrm{DL}}$.

There are two algebraic free functors that are relevant
here.
\begin{itemize}
\item
The left adjoint
\[
\mathbf{R}\:\Comod_{\mathcal{A}_*}
               \to\Comod_{\mathcal{A}_*}^{\mathrm{DL}}
\]
of the forgetful functor
$\Comod_{\mathcal{A}_*}^{\mathrm{DL}}\to\Comod_{\mathcal{A}_*}$;
this is a coproduct $\mathbf{R}=\bigoplus_s\mathbf{R}_s$
where the summand $\mathbf{R}_s$ is expressed in terms
of Dyer-Lashof words of length~$s$.
\item
The left adjoint
\[
\mathbf{U}\:\Comod_{\mathcal{A}_*}^{\mathrm{DL}}
            \to\HA^{\mathcal{A}_*,\,\mathrm{DL}}
\]
of the coalgebra primitives functor
$\mathbf{Pr}\:\HA^{\mathcal{A}_*,\,\mathrm{DL}}
            \to\Comod_{\mathcal{A}_*}^{\mathrm{DL}}$;
this involves the free graded commutative algebra
functor with additional relations coming from the
restriction condition.
\end{itemize}

The structure of $H_*(\mathbb{P}X)$ is given by
the next result.
\begin{thm}\label{thm:H*PX}
If\/ $X$ is cofibrant, then in\/ $\HA^{\mathcal{A}_*,\,\mathrm{DL}}$
there is a natural isomorphism
\[
\mathbf{U}(\mathbf{R}(H_*(X))) \iso H_*(\mathbb{P}X).
\]
\end{thm}

Of course this is an abstract version of a description
in terms of a free algebra on admissible Dyer-Lashof
monomials applied to elements of $H_*(X)$ with suitable
excess conditions; see~\cite{TAQI} for details.

\section{Sample calculations for $p=2$}\label{sec:Sample-2}

In this section we take $p=2$, and assume that all spectra
are localised at~$2$.

Consider the commutative $S$-algebra $S/\!/2$ obtained
as the pushout in the diagram of commutative $S$-algebras
\[
\xymatrix{
& \ar[dl]\mathbb{P}S^0
 \ar@{}[dr]|{\PO}
 \ar@{ >->}[r]^{}\ar@{ >->}[d]
        & \mathbb{P}D^1\ar@{ >->}[d] \\
S &\ar@{->>}[l]_(.4){\sim} \tilde{S} \ar@{ >->}[r] & S/\!/2
}
\]
in which $S^0\xrightarrow{\;\sim\;}S$ is the functorial
cofibrant replacement of $S$ as an $S$-module, $S^0\to D^1\sim*$
is the functorial cofibrant replacement for the collapse
map $S^0\to*$, the diagonal map is induced from a
realisation of the degree~$2$ map $S^0\to S$, and $\tilde{S}$
is defined using the functorial factorisation in the model
category $\mathscr{C}_S$. It follows that $S/\!/2$ is cofibrant
in $\mathscr{C}_S$, and furthermore there is an isomorphism
of commutative $S$-algebras
\[
S/\!/2\iso \tilde{S}\wedge_{\mathbb{P}S^0}\mathbb{P}D^1.
\]
This description allows a calculation of homology using
the K\"unneth spectral sequence. Since the degree~$2$
map induces the trivial map in mod~$2$ ordinary homology,
we can determine $H_*(S/\!/2)=H_*(S/\!/2;\F_2)$ with the
aid of~\cite[theorem~1.7]{BP-Einfinity}. The answer is
\[
H_*(S/\!/2) =
   \F_2[\dlQ^Ix_1:\text{$I$ admissible, $\exc(I)>1$}],
\]
where $x_1\in H_1(S/\!/2)$ satisfies $\Sq^1_*x_1=1$.

Our formulae for the right coaction give
\begin{align*}
\tpsi \dlQ^rx_1 &=
\sum_{1\leq k\leq r}
\dlQ^k(x_1\otimes1 + 1\otimes\zeta_1)[1\otimes\zeta(t)^k]_{t^r} \\
&=
\sum_{1\leq k\leq r}
\biggl[\dlQ^kx_1\otimes\zeta(t)^k
         + 1\otimes(\dlQ^k\zeta_1)\zeta(t)^k\biggr]_{t^r}.
\end{align*}
For example
\begin{align*}
\tpsi \dlQ^2x_1 &=
\dlQ^2x_1\otimes1 + x_1^2\otimes\zeta_1 + 1\otimes(\zeta_1^3 + \zeta_2)
= \dlQ^2x_1\otimes1 + x_1^2\otimes\zeta_1 + 1\otimes\xi_2, \\
\tpsi \dlQ^3x_1 &= \dlQ^3x_1\otimes1 + 1\otimes\dlQ^3\zeta_1
                     = \dlQ^3x_1\otimes1 + 1\otimes\zeta_1^4, \\
\intertext{which give}
\psi\dlQ^2x_1 &=
1\otimes\dlQ^2x_1 + \zeta_1\otimes x_1^2 + \zeta_2\otimes1, \\
\psi\dlQ^3x_1 &= 1\otimes\dlQ^3x_1 + \zeta_1^4\otimes1.
\end{align*}

Using ideas of \cite{DJP:LondonConf,DJP:MSO&MSU}
we will give a description of $H_*(S/\!/2)$ as an
extended $\mathcal{A}_*$-comodule algebra which then
gives an explicit description of $S/\!/2$ as a wedge
of suspensions of $H\F_2$.

First we specify some elements, namely for $s\geq1$,
\[
X_s = \dlQ^{2^s}X_{s-1}
    = \dlQ^{2^s}\dlQ^{2^{s-1}}\cdots\dlQ^2x_1,
\]
where $X_0=x_1$. Notice that the degree of $X_s$
is $|X_s| = 2^{s+1}-1$.

\begin{defn}\label{defn:StrictAllow}
When considering an element of the form $\dlQ^Ix$,
we will say that it is \emph{strictly allowable}
if $I$ is admissible and $\exc(I)>|x|$.
\end{defn}
This is more stringent than the usual notion of
allowable where only $\exc(I)\geq|x|$ is required.
\begin{lem}\label{lem:Xs}
The Dyer-Lashof monomial\/ $\dlQ^rX_s$ is only
strictly allowable if $r=2^{s+1}$.
\end{lem}
\begin{proof}
If $\dlQ^rX_s$ is admissible then $r\leq2^{s+1}$,
while the excess condition holds only if $r>2^{s+1}-1$.
\end{proof}

Next we consider the coaction on these elements.
\begin{prop}\label{prop:psi(Xs)}
In the ring $H_*(S/\!/2)$, the sequence
$X_0,X_1,X_2,\ldots$ is regular, and generates
an ideal $J=(X_s:s\geq0)\lhd H_*(S/\!/2)$. The
left coaction on $X_s$ is given by
\begin{align}
\label{eq:psi(Xs)}
\psi X_s &=
1\otimes X_s + \zeta_1\otimes X_{s-1}^2
   + \zeta_2\otimes X_{s-2}^{2^2} + \cdots
   + \zeta_{s}\otimes X_0^{2^{s}}
   + \zeta_{s+1}\otimes 1 \\
   &\equiv \zeta_{s+1}\otimes 1 \mod{J}.
   \notag
\end{align}
\end{prop}
\begin{proof}
We begin with the formula
\[
\tpsi(X_0) = X_0\otimes 1 +1\otimes\zeta_1
           = X_0\otimes 1 +1\otimes\xi_1.
\]
We will verify by induction on $s$ that
\[
\tpsi X_s =
X_s\otimes 1 + X_{s-1}^2\otimes\xi_1
   + X_{s-2}^{2^2}\otimes\xi_2 + \cdots
   + X_0^{2^{s}}\otimes\xi_{s}
   + 1\otimes\xi_{s+1}.
\]
So assume that this holds for some $s\geq0$. We have
\begin{align*}
\tpsi X_{s+1} = \tpsi\dlQ^{2^{s+1}}X_s
&= (\dlQ^{2^{s+1}-1}\tpsi X_s)(1\otimes\xi_1)
   + \dlQ^{2^{s+1}}\tpsi X_s \\
&= X_s^2\otimes\xi_1 + X_{s-1}^{2^2}\otimes\xi_1^3
   + X_{s-2}^{2^3}\otimes\xi_2^2\xi_1 \\
& \quad\quad\quad
   + \cdots + X_0^{2^{s+1}}\otimes\xi_{s}^2\xi_1
   + 1\otimes\xi_{s+1}^2\xi_1  + X_{s+1}\otimes1 \\
& \quad\quad\quad
   +\dlQ^{2^{s+1}-2}(X_{s-1}^{2})\otimes\dlQ^2\xi_1
   + \dlQ^{2^{s+1}-2^2}(X_{s-2}^{2^2})\otimes\dlQ^{2^2}\xi_2 \\
& \quad\quad\quad
   + \cdots + \dlQ^{2^{s+1}-2^s}(X_0^{2^s})\otimes\dlQ^{2^s}\xi_2
   + 1\otimes\dlQ^{2^{s+1}}\xi_{s+1} \\
&= X_s^2\otimes\xi_1 + X_{s-1}^{2^2}\otimes\xi_1^3
   + X_{s-2}^{2^3}\otimes\xi_2^2\xi_1
   + \cdots + X_0^{2^{s+1}}\otimes\xi_{s}^2\xi_1   \\
& \quad\quad\quad
+ 1\otimes\xi_{s+1}^2\xi_1
+ X_{s+1}\otimes1 + X_{s-1}^{2^2}\otimes\dlQ^2\xi_1
+ X_{s-2}^{2^3}\otimes\dlQ^{2^2}\xi_2 \\
& \quad\quad\quad + \cdots + X_0^{2^{s+1}}\otimes\dlQ^{2^s}\xi_s
        +1\otimes\dlQ^{2^{s+1}}\xi_{s+1} \\
&= X_{s+1}\otimes1 + X_s^2\otimes\xi_1
   + X_{s-1}^{2^2}\otimes(\xi_1^3 + \dlQ^2\xi_1)
   + X_{s-2}^{2^3}\otimes(\xi_2^2\xi_1 + \dlQ^{2^2}\xi_2) \\
& \quad\quad\quad
   + \cdots +
     X_0^{2^{s+1}}\otimes(\xi_{s}^2\xi_1 + \dlQ^{2^s}\xi_s)
   + 1\otimes(\xi_{s+1}^2\xi_1 + \dlQ^{2^{s+1}}\xi_{s+1}) \\
&= X_{s+1}\otimes1 + X_s^2\otimes\xi_1
               + X_{s-1}^{2^2}\otimes\xi_2
               + X_{s-2}^{2^3}\otimes\xi_3 \\
& \quad\quad\quad + \cdots + X_0^{2^{s+1}}\otimes\xi_{s+1}
                  + 1\otimes\xi_{s+2},
\end{align*}
where we make use of Lemma~\ref{lem:DL-xi} in
the last step.
\end{proof}

Now consider the following composition of left
$\mathcal{A}_*$-comodule algebra homomorphisms
\[
\xymatrix{
H_*(S/\!/2) \ar[r]_(.4){\psi}\ar@/^18pt/[rr]^{\bar{\psi}}
& \mathcal{A}_*\otimes H_*(S/\!/2)\ar[r]_(.47){\mathrm{quo}}
& \mathcal{A}_*\otimes H_*(S/\!/2)/J
}
\]
where the second and third terms are extended left
comodules. By Proposition~\ref{prop:psi(Xs)} this
composition is an isomorphism of comodule algebras
\[
H_*(S/\!/2) \xrightarrow[\;\iso\;]{\bar{\psi}}
               \mathcal{A}_*\otimes H_*(S/\!/2)/J
\]
and there is a polynomial subalgebra $P_*\subseteq H_*(S/\!/2)$
with $\bar{\psi}P_* = \F_2\otimes H_*(S/\!/2)/J$.
A standard argument shows that
\[
\pi_*(S/\!/2)\iso P_*
= \Ext^{0,*}_{\mathcal{A}_*}(\F_2,H_*(S/\!/2))
\subseteq H_*(S/\!/2),
\]
and in fact as a spectrum $S/\!/2$ is weakly
equivalent to a wedge of suspensions of $H\F_2$,
and a choice of basis for $P_*$ determines such
a splitting.

We remark that any connective commutative $S$-algebra
$E$ for which $0=2\in\pi_0(E)$ admits a morphism
of commutative $S$-algebras $u\:S/\!/2\to E$. Using
the commutative diagram of $\F_2$-algebras
\[
\xymatrix{
H_*(S/\!/2) \ar@{ >->}[r]^(.45){\psi}\ar[d]_{u_*}
 & \mathcal{A}_*\otimes H_*(S/\!/2)\ar[d]^{I\otimes u_*}  \\
H_*(E) \ar@{ >->}[r]^(.45){\psi} & \mathcal{A}_*\otimes H_*(E)
}
\]
we see that
\[
\psi(u_*X_s) = 1\otimes u_*X_s + \zeta_1\otimes u_*X_{s-1}^2
   + \zeta_2\otimes u_*X_{s-2}^{2^2} + \cdots
   + \zeta_{s}\otimes u_*X_1^{2^{s}}
   + \zeta_{s+1}\otimes 1
\]
so the $u_*X_s$ is a sequence of algebraically independent
elements. It follows that there is an isomorphism of
$\mathcal{A}_*$-comodule algebras
\[
H_*(E) \iso \mathcal{A}_*\otimes H_*(E)/(u_*X_s:s\geq0),
\]
so $H_*(E)$ is also an extended comodule and $E$ is
weakly equivalent to a wedge of suspensions of $H\F_2$.
This gives a different approach to proving Steinberger's
result~\cite[theorem~III.4.1]{LNM1176} which potentially
contains more information on the multiplicative structure
of the splitting.

Rolf Hoyer has pointed out some explicit formulae for
primitives in $H_*(S/\!/2)$ and thus for families of
polynomial generators for $\pi_*(S/\!/2)$.

\section{Sample calculations for odd primes}
\label{sec:Sample-odd}

Now we assume that $p$ is an odd prime and that all
spectra are localised at $p$. There are similarities
to the $2$-primary case, although some of the details
are slightly more complicated.

Consider the commutative $S$-algebra $S/\!/p$ which is
the pushout in the diagram of commutative $S$-algebras
\[
\xymatrix{
& \ar[dl]\mathbb{P}S^0
 \ar@{}[dr]|{\PO}
 \ar@{ >->}[r]^{}\ar@{ >->}[d]
        & \mathbb{P}D^1\ar@{ >->}[d] \\
S &\ar@{->>}[l]_(.4){\sim} \tilde{S} \ar@{ >->}[r] & S/\!/p
}
\]
where the notation is similar to that in the case
$p=2$. Then $S/\!/p$ is cofibrant in $\mathscr{C}_S$,
and there is an isomorphism of commutative $S$-algebras
\[
S/\!/p\iso \tilde{S}\wedge_{\mathbb{P}S^0}\mathbb{P}D^1.
\]
Since the degree $p$ map induces the trivial map in mod~$p$
ordinary homology, $H_*(S/\!/p)=H_*(S/\!/p;\F_p)$ can be
determined by methods of~\cite[theorem~1.7]{BP-Einfinity}.
The answer is a free graded commutative algebra
\[
H_*(S/\!/p) =
\F_p\langle\dlQ^Ix_1:\text{$I$ admissible, $\exc(I)>1$}\rangle,
\]
where $x_1\in H_1(S/\!/p)$ satisfies $\beta x_1=1$.

We define two sequences of elements, beginning with $X_0=x_1$
and $Y_0=1$,
\[
X_s = \dlQ^{p^{s-1}}X_{s-1}
  = \dlQ^{p^{s-1}}\dlQ^{p^{s-2}}\cdots\dlQ^{p}\dlQ^1x_1,
  \;
Y_s = \beta\dlQ^{p^{s-1}}X_{s-1}
  = \beta\dlQ^{p^{s-1}}\dlQ^{p^{s-2}}\cdots\dlQ^{p}\dlQ^1x_1.
\]
Notice that the degrees of these elements are $|X_s|=2p^s-1$
and $|Y_s|=2(p^s-1)$.

We will again use the terminology of strictly allowable
introduced in Definition~\ref{defn:StrictAllow}.

\begin{lem}\label{lem:XsYs-p}
Let\/ $r\geq1$. Then the Dyer-Lashof monomial $\dlQ^rX_s$
is strictly allowable only if\/ $r=p^{s}$, while
$\dlQ^rY_s$ is never strictly allowable.
\end{lem}
\begin{proof}
If $\dlQ^rX_s=\dlQ^r\dlQ^{p^{s-1}}X_{s-1}$ is admissible
then $r\leq p^{s}$, while the required excess condition
is $2r>2p^{s}-1$.

If $\dlQ^rY_s=\dlQ^r\beta\dlQ^{p^{s-1}}X_{s-1}$ is
admissible then $r<p^{s}$, while the excess condition
required for it to be strictly allowable is $2r>2(p^s-1)$.
Clearly these conditions are contradictory.
\end{proof}

\begin{prop}\label{prop:psi(Xs,Ys)-p}
The left coaction on $X_s$ and $Y_s$ is given by
\begin{subequations}\label{eq:psi(Xs,Ys)-p}
\begin{align}\label{eq:psi(Xs)-p}
\psi X_s &=
1\otimes X_s
   + \btau_0\otimes Y_{s}
   + \btau_1\otimes Y_{s-1}^p
   + \btau_2\otimes Y_{s-2}^{p^2} + \cdots
   + \btau_{s-1}\otimes Y_1^{p^{s-1}}
   + \btau_{s}\otimes 1, \\
\label{eq:psi(Ys)-p}
\psi Y_s &= 1\otimes Y_s + \zeta_1\otimes Y_{s-1}^p
   + \zeta_2\otimes Y_{s-2}^{p^2} + \cdots
   + \zeta_{s-1}\otimes Y_1^{p^{s-1}}
   + \zeta_{s}\otimes 1.
\end{align}
\end{subequations}
\end{prop}
\begin{proof}
Translating the formulae into statements about
the right coaction we must prove that the
following equations are satisfied for every~$s$:
\begin{align*}
\tpsi X_s &=
X_s\otimes1
   + Y_{s}\otimes\tau_0
   + Y_{s-1}^p\otimes\tau_1
   + Y_{s-2}^{p^2}\otimes\tau_2 + \cdots
   + Y_1^{p^{s-1}}\otimes\tau_{s-1}
   + 1\otimes\tau_{s}, \\
\tpsi Y_s &= Y_s\otimes1 + Y_{s-1}^p\otimes\xi_1
   + Y_{s-2}^{p^2}\otimes\xi_2 + \cdots
   + Y_1^{p^{s-1}}\otimes\xi_{s-1}
   + 1\otimes\xi_{s}.
\end{align*}
Assuming these are true for some $s$, we have
\begin{align*}
\tpsi X_{s+1} =&\; \tpsi\dlQ^{p^{s}}X_{s} \\
=&\; \dlQ^{p^{s}}(\tpsi X_{s})
     - \beta\dlQ^{p^{s}}(\tpsi X_{s})\btau_0 \\
=&\;
\dlQ^{p^{s}}
\biggl(X_s\otimes1 + Y_{s}\otimes\tau_0 + Y_{s-1}^p\otimes\tau_1
   + Y_{s-2}^{p^2}\otimes\tau_2 + \cdots
   + Y_1^{p^{s-1}}\otimes\tau_{s-1}
   + 1\otimes\tau_{s}\biggr) \\
&  + \beta\dlQ^{p^{s}}\biggl(X_s\otimes1
   + Y_{s}\otimes\tau_0
   + Y_{s-1}^p\otimes\tau_1
   + Y_{s-2}^{p^2}\otimes\tau_2
   + \cdots
   + Y_1^{p^{s-1}}\otimes\tau_{s-1}
   + 1\otimes\tau_{s}\biggr)\tau_0 \\
=&\;
\biggl(\dlQ^{p^{s}}X_s\otimes1
  + (\dlQ^{p^{s}-1}Y_{s})\otimes\dlQ^1\tau_0
  + (\dlQ^{p^{s-1}-1}Y_{s-1})^p\otimes\dlQ^{p}\tau_1 \\
& + (\dlQ^{p^{s-2}-1}Y_{s-2})^{p^2}\otimes\dlQ^{p^2}\tau_2
+ \cdots
+ (\dlQ^{p-1}Y_{1})^{p^{s-1}}\otimes\dlQ^{p^{s-1}}\tau_{s-1}
   + 1\otimes\dlQ^{p^{s}}\tau_{s}\biggr) \\
&  + \biggl(\beta\dlQ^{p^{s}}X_s\otimes1
   + (\dlQ^{p^{s}-1}Y_{s})\otimes\beta\dlQ^{1}\tau_0
   + (\dlQ^{p^{s-1}-1}Y_{s-1})^p\otimes\beta\dlQ^{p}\tau_1 \\
&  + (\dlQ^{p^{s-2}-1}Y_{s-2})^{p^2}\otimes\beta\dlQ^{p^2}\tau_2
   + \cdots
   + (\dlQ^{p-1}Y_1)^{p^{s-1}}\otimes\beta\dlQ^{p^{s-1}}\tau_{s-1}
   + 1\otimes\beta\dlQ^{p^{s}}\tau_{s}\biggr)\tau_0 \\
=&\;
\biggl(X_{s+1}\otimes1 + Y_{s}^{p}\otimes(\tau_1 - \tau_0\xi_1)
+ Y_{s-1}^{p^2}\otimes(\tau_2 - \tau_0\xi_2)
+ Y_{s-2}^{p^3}\otimes(\tau_3 - \tau_0\xi_3) \\
& \quad\quad\quad\quad\quad\quad\quad\quad\quad\quad\quad\quad\quad\quad\quad
   + \cdots + Y_{1}^{p^{s}}\otimes(\tau_{s} - \tau_0\xi_{s})
   + 1\otimes(\tau_{s+1} - \tau_0\xi_{s+1})\biggr) \\
&  + \biggl(Y_{s+1}\otimes1
   + Y_{s}^{p}\otimes\xi_1
   + Y_{s-1}^{p^2}\otimes\xi_2
   + Y_{s-2}^{p^3}\otimes\xi_3
   + \cdots + Y_1^{p^{s}}\otimes\xi_{s}
   + 1\otimes\xi_{s+1}\biggr)\tau_0 \\
=&\;
X_{s+1}\otimes1 + Y_{s+1}\otimes\tau_0
+ Y_{s}^p\otimes\tau_1
+ Y_{s-1}^{p^2}\otimes\tau_2
+ \cdots +
Y_{1}^{p^{s}}\otimes\tau_{s}
+ 1\otimes\tau_{s+1}.
\end{align*}
A similar calculation shows that
\[
\tpsi Y_{s+1} = Y_{s+1}\otimes1 + Y_{s}^p\otimes\xi_1
 + Y_{s-1}^{p^2}\otimes\xi_2 + \cdots
 + Y_2^{p^{s-1}}\otimes\xi_{s-1} + Y_1^{p^{s}}\otimes\xi_{s}
 + 1\otimes\xi_{s+1}.
\]
The result follows by Induction.
\end{proof}

Let $J=(X_s,Y_{s+1}:s\geq0)\lhd H_*(S/\!/p)$ be the
ideal generated by the elements $X_s,Y_s$. As happens
for the prime~$2$, the following composition of left
$\mathcal{A}_*$-comodule algebra homomorphisms
\[
\xymatrix{
H_*(S/\!/p) \ar[r]_(.4){\psi}\ar@/^18pt/[rr]^{\bar{\psi}}
& \mathcal{A}_*\otimes H_*(S/\!/p)\ar[r]_(.47){\mathrm{quo}}
& \mathcal{A}_*\otimes H_*(S/\!/p)/J
}
\]
is an isomorphism, where the second and third terms are
extended left comodules. Here $H_*(S/\!/p)/J$ is a free
graded commutative algebra since the generators $X_s$
and $Y_s$ are amongst the generators of the free graded
commutative algebra $H_*(S/\!/p)$. There is a subalgebra
$P_*\subseteq H_*(S/\!/p)$ which is identified with
$\F_p\otimes H_*(S/\!/p)/J$ under the isomorphism $\bar{\psi}$,
i.e., $\bar{\psi}P_* = \F_p\otimes H_*(S/\!/p)/J$. A
standard argument shows the spectrum $S/\!/p$ is equivalent
to a wedge of suspensions of $H\F_p$. As we saw in the
$2$-primary case, this leads to a proof of
Steinberger's result~\cite[theorem~III.4.1]{LNM1176}.

\end{document}